\documentclass[executivepaper,oneside]{article}
\usepackage{amsfonts,amsmath,amssymb}
\usepackage{geometry}

\newtheorem{theorem}{Theorem}[section]

\newtheorem{corollary}[theorem]{Corollary}

\newtheorem{definition}{Definition}[section]

\newtheorem{lemma}{Lemma}[section]

\newtheorem{remark}[theorem]{Remark}

\numberwithin{equation}{section}
\newenvironment{proof}[1][Proof]{\noindent\textbf{#1.} }{\ \rule{0.5em}{0.5em}}
\allowdisplaybreaks

\begin{document}

\title{{\small SOME GRUSS-TYEP INEQUALITIES USING GENERALIZED KATUGAMPOLA
FRACTIONAL INTEGRAL}}
\author{ {\small Tariq A. Aljaaidi, Deepak B. Pachpatte}}
\date{}
\maketitle

\begin{abstract}
The main objective of this paper is to obtain generalization of some
Gruss-type inequalities in case of functional bounds by using a generalized
Katugampola fractional integral.
\end{abstract}

\textbf{Keywords: }Gruss inequality; generalized fractional integral.

\section{\protect\small INTRODUCTION:}

In 1935, G. Gr\"{u}ss proved the renowned integral inequality \cite{aa} (see
also \cite{cc}):%
\begin{eqnarray}
&&\left\vert \frac{1}{b-a}\int_{a}^{b}f\left( x\right) g\left( x\right) dx-%
\frac{1}{\left( b-a\right) ^{2}}\int_{a}^{b}f\left( x\right)
dx\int_{a}^{b}g\left( x\right) dx\right\vert  \label{inq55} \\
&&\left. \leq \frac{1}{4}\left( M-m\right) \left( P-p\right) \right. ,
\notag
\end{eqnarray}%
where $f,g$ are two integrable functions on $\left[ a,b\right] ,$ satisfying
the conditions \ \ \

$\ \ \ \ m\leq f\left( x\right) \leq M$ \ \ and \ \ $p\leq g\left( x\right)
\leq P$ \ \ for all $x\in \left[ a,b\right] ,$ \ \ $m,M,p,P\in
%TCIMACRO{\U{211d} }%
%BeginExpansion
\mathbb{R}
%EndExpansion
.$

\ \ In recent years, the inequalities are getting to play a very vital role
in all mathematical fields, especially after the creation of fractional
calculus which gave rise to several results and important theories in
mathematics, engineering, physics, and other fields of science.

A remarkably large number inequalities of above type involving the\ special
fractional integral (such as the Liouville, Riemann--Liouville, Erd\'{e}%
lyi-Kober, Katugampola, Hadamard and Weyl types) have been investigated by
many researchers and received considerable attention to it, see( \cite{uu},
\cite{kk}, \cite{ll}, \cite{rr}, \cite{mm}, \cite{tt}, \cite{nn}, \cite{qq},
\cite{oo}, \cite{ss}, \cite{pp}).

Gruss type inequality which has some important applications in a number of
mathematical fields, like an integral arithmetic mean, difference equations
and h-integral arithmetic mean (see \cite{dd}, \cite{bb}).

Dahmani et al. \cite{ee}, in (2010), proved the following fractional version
inequality by using Riemann--Liouville fractional integral%
\begin{equation}
\left\vert \frac{x^{\alpha }}{\Gamma \left( \alpha +1\right) }\mathcal{J}%
^{\alpha }\left( fg\right) \left( x\right) -\mathcal{J}^{\alpha }f\left(
x\right) \mathcal{J}^{\alpha }g\left( x\right) \right\vert \leq \left( \frac{%
x^{\alpha }}{\Gamma \left( \alpha +1\right) }\right) ^{2}\left( M-m\right)
\left( P-p\right) ,  \label{INQ17}
\end{equation}%
for one parameter, and%
\begin{eqnarray}
&&\left( \frac{x^{\alpha }}{\Gamma \left( \alpha +1\right) }\mathcal{J}%
^{\beta }\left( fg\right) \left( x\right) -\mathcal{J}^{\alpha }f\left(
x\right) \mathcal{J}^{\beta }g\left( x\right) \right.  \notag \\
&&\left. +\frac{x^{\beta }}{\Gamma \left( \beta +1\right) }\mathcal{J}%
^{\alpha }\left( fg\right) \left( x\right) -\mathcal{J}^{\beta }f\left(
x\right) \mathcal{J}^{\alpha }g\left( x\right) \right) ^{2}  \notag \\
&&\left. \text{ }\leq \left[ \left( M\frac{x^{\alpha }}{\Gamma \left( \alpha
+1\right) }-\mathcal{J}^{\alpha }f\left( x\right) \right) \left( \mathcal{J}%
^{\beta }f\left( x\right) -m\frac{x^{\beta }}{\Gamma \left( \beta +1\right) }%
\right) \right. \right.  \notag \\
&&\left. \text{ }+\left( \mathcal{J}^{\alpha }f\left( x\right) -m\frac{%
x^{\alpha }}{\Gamma \left( \alpha +1\right) }\right) \left( M\frac{x^{\beta }%
}{\Gamma \left( \beta +1\right) }-\mathcal{J}^{\beta }f\left( x\right)
\right) \right]  \label{inq56} \\
&&\text{ \ }\times \left[ \left( P\frac{x^{\alpha }}{\Gamma \left( \alpha
+1\right) }-\mathcal{J}^{\alpha }g\left( x\right) \right) \left( \mathcal{J}%
^{\beta }g\left( x\right) -p\frac{x^{\beta }}{\Gamma \left( \beta +1\right) }%
\right) \right.  \notag \\
&&\left. \text{ }+\left( \mathcal{J}^{\alpha }g\left( x\right) -p\frac{%
x^{\alpha }}{\Gamma \left( \alpha +1\right) }\right) \left( P\frac{x^{\beta }%
}{\Gamma \left( \beta +1\right) }-\mathcal{J}^{\beta }g\left( x\right)
\right) \right] ,  \notag
\end{eqnarray}%
for two parameters, where $f,g$ are two integrable functions on $\left[
0,\infty \right) ,$ satisfying the conditions%
\begin{equation}
\text{ \ \ \ }m\leq f\left( x\right) \leq M\text{ \ \ }and\text{ \ \ }p\leq
g\left( x\right) \leq P\text{ \ \ }for\text{ }all\text{ }x\in \left[
0,\infty \right) ,\text{ \ \ }m,M,p,P\in
%TCIMACRO{\U{211d} }%
%BeginExpansion
\mathbb{R}
%EndExpansion
.  \label{INQ18}
\end{equation}

In (2014), Tariboon et al. \cite{ff}, replaced the constants which appeared
as bounds of the functions $f$\ and $g$ by four integrable functions on $%
\left[ 0,\infty \right) $, as\ \ $\varphi _{1}\left( x\right) \leq f\left(
x\right) \leq \varphi _{2}\left( x\right) $ \ \ $and$ \ \ $\psi _{1}\left(
x\right) \leq g\left( x\right) \leq \psi _{2}\left( x\right) ,$ they
obtained inequality%
\begin{equation*}
\left\vert \frac{x^{\alpha }}{\Gamma \left( \alpha +1\right) }\mathcal{J}%
^{\alpha }\left( fg\right) \left( x\right) -\mathcal{J}^{\alpha }f\left(
x\right) \mathcal{J}^{\alpha }g\left( x\right) \right\vert \leq \sqrt{%
T\left( f,\varphi _{1},\varphi _{2}\right) T\left( g,\psi _{1},\psi
_{2}\right) },
\end{equation*}%
where $T\left( u,v,w\right) $ is defined by%
\begin{eqnarray*}
T\left( u,v,w\right) &=&\left( \mathcal{J}^{\alpha }\omega \left( x\right) -%
\mathcal{J}^{\alpha }u\left( x\right) \right) \left( \mathcal{J}^{\alpha
}u\left( x\right) -\mathcal{J}^{\alpha }v\left( x\right) \right) \\
&&+\frac{x^{\alpha }}{\Gamma \left( \alpha +1\right) }\mathcal{J}^{\alpha
}\left( uv\right) \left( x\right) -\mathcal{J}^{\alpha }u\left( x\right)
\mathcal{J}^{\alpha }v\left( x\right) \\
&&+\frac{x^{\alpha }}{\Gamma \left( \alpha +1\right) }\mathcal{J}^{\alpha
}\left( u\omega \right) \left( x\right) -\mathcal{J}^{\alpha }u\left(
x\right) \mathcal{J}^{\alpha }\omega \left( x\right) \\
&&-\frac{x^{\alpha }}{\Gamma \left( \alpha +1\right) }\mathcal{J}^{\alpha
}\left( v\omega \right) \left( x\right) +\mathcal{J}^{\alpha }v\left(
x\right) \mathcal{J}^{\alpha }\omega \left( x\right) .
\end{eqnarray*}

Motivated from above mentioned results, our purpose in this paper is to
establish some new results on Gruss-type inequalities in case of functional
bounds using the generalized Katugampola fractional integral.

\section{\protect\small PRELIMINARIES:}

\ \ \ In this section, we give some definitions and properties available in
literature that will be used in our paper, for more details (see \cite{hh},
\cite{jj}, \cite{gg}).

\begin{definition}
Consider the space $X_{c}^{p}\left( a,b\right) \left( c\in
%TCIMACRO{\U{211d} }%
%BeginExpansion
\mathbb{R}
%EndExpansion
,1\leq p\leq \infty \right) $, of those complex valued Lebesgue measurable
functions $v$ on $\left( a,b\right) $ for which the norm $\left\Vert
v\right\Vert _{X_{c}^{p}}<\infty $, such that%
\begin{equation*}
\left\Vert v\right\Vert _{X_{c}^{p}}=\left( \int_{x}^{b}\left\vert
x^{c}v\right\vert ^{p}\frac{dx}{x}\right) ^{\frac{1}{p}},\text{ \ \ \ }%
\left( 1\leq p<\infty \right)
\end{equation*}%
and \
\begin{equation*}
\left\Vert v\right\Vert _{X_{c}^{\infty }}=\sup ess_{x\in \left( a,b\right) }
\left[ x^{c}\left\vert v\right\vert \right] .
\end{equation*}%
In particular, when $c=1/p,$ the space $X_{c}^{p}\left( a,b\right) $
coincides with the space $L^{p}\left( a,b\right) .$
\end{definition}

\begin{definition}
\label{dif1}The left- and right-sided fractional integrals of a function $v$
where $v\in $ $X_{c}^{p}\left( a,b\right) ,$ $\alpha >0,$ and $\beta ,\rho
,\eta ,k\in
%TCIMACRO{\U{211d} }%
%BeginExpansion
\mathbb{R}
%EndExpansion
,$ are defined respectively by%
\begin{equation}
\text{ }^{\rho }\mathcal{J}_{a+;\eta ,k}^{\alpha ,\beta }v\left( x\right) =%
\frac{\rho ^{1-\beta }x^{k}}{\Gamma \left( \alpha \right) }\int_{a}^{x}\frac{%
\tau ^{\rho \left( \eta +1\right) -1}}{\left( x^{\rho }-\tau ^{\rho }\right)
^{1-\alpha }}v\left( \tau \right) d\tau ,\text{ \ \ \ \ \ \ \ \ }0\leq
a<x<b\leq \infty ,  \label{oper2}
\end{equation}%
and%
\begin{equation}
\text{ }^{\rho }\mathcal{J}_{b-;\eta ,k}^{\alpha ,\beta }v\left( x\right) =%
\frac{\rho ^{1-\beta }x^{\rho \eta }}{\Gamma \left( \alpha \right) }%
\int_{x}^{b}\frac{\tau ^{k+\rho -1}}{\left( \tau ^{\rho }-x^{\rho }\right)
^{1-\alpha }}v\left( \tau \right) d\tau ,\text{ \ \ \ \ \ \ \ \ }0\leq
a<x<b\leq \infty ,  \label{eq1}
\end{equation}%
if the integral exist.
\end{definition}

To present and discuss our new results in this paper we use the left-sided
fractional integrals, the right sided fractional can be proved similarly,
also we consider $a=0,$ in (\ref{oper2}), to obtain%
\begin{equation*}
^{\rho }\mathcal{I}_{\eta ,k}^{\alpha ,\beta }v\left( x\right) =\frac{\rho
^{1-\beta }x^{k}}{\Gamma \left( \alpha \right) }\int_{0}^{x}\frac{\tau
^{\rho \left( \eta +1\right) -1}}{\left( x^{\rho }-\tau ^{\rho }\right)
^{1-\alpha }}v\left( \tau \right) d\tau .
\end{equation*}

The above fractional integral has the following Composition (index) formulae%
\begin{equation*}
\text{ }^{\rho }\mathcal{J}_{a+;\eta _{1},k_{1}}^{\alpha _{1},\beta _{1}}%
\text{ }^{\rho }\mathcal{J}_{a+;\eta _{2},-\rho \eta _{1}}^{\alpha
_{2},\beta _{2}}v=\text{ }^{\rho }\mathcal{J}_{a+;\eta _{2},k_{1}}^{\alpha
_{1}+\alpha _{2},\beta _{1}+\beta _{2}}v,
\end{equation*}%
\begin{equation*}
\text{ }^{\rho }\mathcal{J}_{b-;\eta _{1},-\rho \eta _{2}}^{\alpha
_{1},\beta _{1}}\text{ }^{\rho }\mathcal{J}_{b-;\eta _{2},k_{2}}^{\alpha
_{2},\beta _{2}}v=\text{ }^{\rho }\mathcal{J}_{a+;\eta _{1},k_{2}}^{\alpha
_{1}+\alpha _{2},\beta _{1}+\beta _{2}}v.
\end{equation*}

For the convenience of establishing our results we define the following
function as in \cite{gg}: let $x>0,$ $\alpha >0,$ $\rho ,k,\beta ,\eta \in
%TCIMACRO{\U{211d} }%
%BeginExpansion
\mathbb{R}
%EndExpansion
$, then
\begin{equation*}
\Lambda _{x,k}^{\rho ,\beta }\left( \alpha ,\eta \right) =\frac{\Gamma
\left( \eta +1\right) }{\Gamma \left( \eta +\alpha +1\right) }\rho ^{-\beta
}x^{k+\rho \left( \eta +\alpha \right) }.
\end{equation*}%
\ \

If $\eta =0,$ $a=0,$ $k=0,$ and taking the limit $\rho \rightarrow 1,$ the
Definition (\ref{dif1}) reduce to Liouville fractional integral and if $\eta
=0,$ $k=0,$ and taking the limit $\rho \rightarrow 1,$ we can get
Riemann-Liouville fractional integral, it is reduce to Weyl fractional
integral, if $\eta =0,$ $a=-\infty ,$ $k=0,$ and taking the limit $\rho
\rightarrow 1,$ for Erd\'{e}lyi-Kober fractional integral, we put $\beta =0,$
$k=-\rho \left( \alpha +\eta \right) ,$ we can also getting Katugampola
fractional integral by taking $\beta =\alpha ,$ $k=0,$ $\eta =0,$ And
finally Hadamard fractional integral if $\beta =\alpha ,$ $k=0,$ $\eta
=0^{+},$ and taking the limit $\rho \rightarrow 1.$

The Definition (\ref{dif1}) is more generalized and can be reduce to six
cases by change its parameters with appropriate choice.

\section{\protect\small Main Results:}

Now, we give our main results on Gruss type inequality in case of functional
bounds.

\begin{theorem}
\label{thm5} Let $v$ be an integrable function on $\left[ 0,\infty \right) .$
Assume that there exist two integrable functions $z_{1},z_{2}$ on $\left[
0,\infty \right) $ such that \
\begin{equation}
z_{1}\left( x\right) \leq v\left( x\right) \leq z_{2}\left( x\right) \ \ \ \
\ \ \forall x\in \left[ 0,\infty \right) .  \label{cond1}
\end{equation}%
Then, for all $x>0,$ $\alpha >0,$ $\rho >0,$ $\delta >0,$ $\beta ,\eta
,k,\lambda \in
%TCIMACRO{\U{211d} }%
%BeginExpansion
\mathbb{R}
%EndExpansion
,$ we have

\begin{eqnarray}
&&^{\rho }\mathcal{J}_{\eta ,k}^{\alpha ,\beta }z_{2}\left( x\right) \text{ }%
^{\rho }\mathcal{J}_{\eta ,k}^{\delta ,\lambda }v\left( x\right) +\text{ }%
^{\rho }\mathcal{J}_{\eta ,k}^{\alpha ,\beta }v\left( x\right) \text{ }%
^{\rho }\mathcal{J}_{\eta ,k}^{\delta ,\lambda }z_{1}\left( x\right)
\label{inq66} \\
&&\left. \text{ }\geq \text{ }^{\rho }\mathcal{J}_{\eta ,k}^{\alpha ,\beta
}v\left( x\right) \text{ }^{\rho }\mathcal{J}_{\eta ,k}^{\delta ,\lambda
}v\left( x\right) +\text{ }^{\rho }\mathcal{J}_{\eta ,k}^{\alpha ,\beta
}z_{2}\left( x\right) \text{ }^{\rho }\mathcal{J}_{\eta ,k}^{\delta ,\lambda
}z_{1}\left( x\right) \right.  \notag
\end{eqnarray}
\end{theorem}

\begin{proof}
From the condition (\ref{cond1}), for all $\tau \geq 0,$ $\sigma \geq 0,$ we
have
\begin{equation*}
\left( \ v\left( \sigma \right) -z_{1}\left( \sigma \right) \right) \left(
z_{2}\left( \tau \right) -v\left( \tau \right) \right) \geq 0.
\end{equation*}

Therefore
\begin{equation}
v\left( \sigma \right) z_{2}\left( \tau \right) +z_{1}\left( \sigma \right)
v\left( \tau \right) \geq v\left( \sigma \right) v\left( \tau \right)
+z_{1}\left( \sigma \right) z_{2}\left( \tau \right) .  \label{inq 2}
\end{equation}

Multiplying both sides of (\ref{inq 2}) by $\frac{\rho ^{1-\beta }x^{k}}{%
\Gamma \left( \alpha \right) }\frac{\tau ^{\rho \left( \eta +1\right) -1}}{%
\left( x^{\rho }-\tau ^{\rho }\right) ^{1-\alpha }},$ where $\tau \in \left(
0,x\right) $ and integrating with respect to $\tau $ over $\left( 0,x\right)
$, we get

\begin{eqnarray*}
&&v\left( \sigma \right) \frac{\rho ^{1-\beta }x^{k}}{\Gamma \left( \alpha
\right) }\int_{0}^{x}\frac{\tau ^{\rho \left( \eta +1\right) -1}}{\left(
x^{\rho }-\tau ^{\rho }\right) ^{1-\alpha }}z_{2}\left( \tau \right) d\tau \\
&&+z_{1}\left( \sigma \right) \frac{\rho ^{1-\beta }x^{k}}{\Gamma \left(
\alpha \right) }\int_{0}^{x}\frac{\tau ^{\rho \left( \eta +1\right) -1}}{%
\left( x^{\rho }-\tau ^{\rho }\right) ^{1-\alpha }}v\left( \tau \right)
d\tau \text{ } \\
&&\left. \text{ }\geq v\left( \sigma \right) \frac{\rho ^{1-\beta }x^{k}}{%
\Gamma \left( \alpha \right) }\int_{0}^{x}\frac{\tau ^{\rho \left( \eta
+1\right) -1}}{\left( x^{\rho }-\tau ^{\rho }\right) ^{1-\alpha }}v\left(
\tau \right) d\tau \right. \\
&&\text{ \ }+z_{1}\left( \sigma \right) \frac{\rho ^{1-\beta }x^{k}}{\Gamma
\left( \alpha \right) }\int_{0}^{x}\frac{\tau ^{\rho \left( \eta +1\right)
-1}}{\left( x^{\rho }-\tau ^{\rho }\right) ^{1-\alpha }}z_{2}\left( \tau
\right) d\tau ,
\end{eqnarray*}%
so we have%
\begin{equation}
v\left( \sigma \right) \text{ }^{\rho }\mathcal{J}_{\eta ,k}^{\alpha ,\beta
}z_{2}\left( x\right) +z_{1}\left( \sigma \right) \text{ }^{\rho }\mathcal{J}%
_{\eta ,k}^{\alpha ,\beta }v\left( x\right) \geq v\left( \sigma \right)
\text{ }^{\rho }\mathcal{J}_{\eta ,k}^{\alpha ,\beta }v\left( x\right)
+z_{1}\left( \sigma \right) \text{ }^{\rho }\mathcal{J}_{\eta ,k}^{\alpha
,\beta }z_{2}\left( x\right) .  \label{INQ3}
\end{equation}%
Multiplying both sides of (\ref{INQ3}) by $\frac{\rho ^{1-\beta }x^{k}}{%
\Gamma \left( \delta \right) }\frac{\sigma ^{\rho \left( \eta +1\right) -1}}{%
\left( x^{\rho }-\sigma ^{\rho }\right) ^{1-\delta }},$ where $\sigma \in
\left( 0,x\right) $, we obtain%
\begin{eqnarray}
&&\frac{\rho ^{1-\beta }x^{k}}{\Gamma \left( \delta \right) }\frac{\sigma
^{\rho \left( \eta +1\right) -1}}{\left( x^{\rho }-\sigma ^{\rho }\right)
^{1-\delta }}v\left( \sigma \right) \text{ }^{\rho }\mathcal{J}_{\eta
,k}^{\alpha ,\beta }z_{2}\left( x\right)  \notag \\
&&+\frac{\rho ^{1-\beta }x^{k}}{\Gamma \left( \delta \right) }\frac{\sigma
^{\rho \left( \eta +1\right) -1}}{\left( x^{\rho }-\sigma ^{\rho }\right)
^{1-\delta }}z_{1}\left( \sigma \right) \text{ }^{\rho }\mathcal{J}_{\eta
,k}^{\alpha ,\beta }v\left( x\right) \text{ \ \ \ \ \ \ \ \ \ \ \ \ \ \ \ \
\ \ \ \ \ \ \ \ \ \ \ \ \ \ }  \label{ineq1} \\
&&\left. \text{ }\geq \frac{\rho ^{1-\beta }x^{k}}{\Gamma \left( \delta
\right) }\frac{\sigma ^{\rho \left( \eta +1\right) -1}}{\left( x^{\rho
}-\sigma ^{\rho }\right) ^{1-\delta }}v\left( \sigma \right) \text{ }^{\rho }%
\mathcal{J}_{\eta ,k}^{\alpha ,\beta }\right.  \notag \\
&&\text{ \ }+\frac{\rho ^{1-\beta }x^{k}}{\Gamma \left( \delta \right) }%
\frac{\sigma ^{\rho \left( \eta +1\right) -1}}{\left( x^{\rho }-\sigma
^{\rho }\right) ^{1-\delta }}z_{1}\left( \sigma \right) \text{ }^{\rho }%
\mathcal{J}_{\eta ,k}^{\alpha ,\beta }z_{2}\left( x\right) .  \notag
\end{eqnarray}%
Integrating both sides of (\ref{ineq1}) with respect to $\sigma $ over $%
\left( 0,x\right) $, we get%
\begin{eqnarray*}
&&\text{ }^{\rho }\mathcal{J}_{\eta ,k}^{\alpha ,\beta }z_{2}\left( x\right)
\frac{\rho ^{1-\lambda }x^{k}}{\Gamma \left( \delta \right) }\int_{0}^{x}%
\frac{\sigma ^{\rho \left( \eta +1\right) -1}}{\left( x^{\rho }-\sigma
^{\rho }\right) ^{1-\delta }}v\left( \sigma \right) d\sigma \\
&&+\text{ }^{\rho }\mathcal{J}_{\eta ,k}^{\alpha ,\beta }v\left( x\right)
\frac{\rho ^{1-\lambda }x^{k}}{\Gamma \left( \delta \right) }\int_{0}^{x}%
\frac{\sigma ^{\rho \left( \eta +1\right) -1}}{\left( x^{\rho }-\sigma
^{\rho }\right) ^{1-\delta }}z_{1}\left( \sigma \right) d\sigma \text{ \ \ \
\ \ \ \ \ \ \ \ \ \ \ \ \ \ \ \ \ \ \ \ } \\
&&\left. \text{ }\geq \text{ }^{\rho }\mathcal{J}_{\eta ,k}^{\alpha ,\beta
}v\left( x\right) \frac{\rho ^{1-\lambda }x^{k}}{\Gamma \left( \delta
\right) }\int_{0}^{x}\frac{\sigma ^{\rho \left( \eta +1\right) -1}}{\left(
x^{\rho }-\sigma ^{\rho }\right) ^{1-\delta }}v\left( \sigma \right) d\sigma
\right. \\
&&\text{ \ }+\text{ }^{\rho }\mathcal{J}_{\eta ,k}^{\alpha ,\beta
}z_{2}\left( x\right) \frac{\rho ^{1-\lambda }x^{k}}{\Gamma \left( \delta
\right) }\int_{0}^{x}\frac{\sigma ^{\rho \left( \eta +1\right) -1}}{\left(
x^{\rho }-\sigma ^{\rho }\right) ^{1-\delta }}z_{1}\left( \sigma \right)
d\sigma .
\end{eqnarray*}%
Hence%
\begin{eqnarray*}
&&\text{ }^{\rho }\mathcal{J}_{\eta ,k}^{\alpha ,\beta }z_{2}\left( x\right)
\text{ }^{\rho }\mathcal{J}_{\eta ,k}^{\delta ,\lambda }v\left( x\right) +%
\text{ }^{\rho }\mathcal{J}_{\eta ,k}^{\alpha ,\beta }v\left( x\right) \text{
}^{\rho }\mathcal{J}_{\eta ,k}^{\delta ,\lambda }z_{1}\left( x\right) \\
&&\left. \text{ }\geq \text{ }^{\rho }\mathcal{J}_{\eta ,k}^{\alpha ,\beta
}v\left( x\right) \text{ }^{\rho }\mathcal{J}_{\eta ,k}^{\delta ,\lambda
}v\left( x\right) +\text{ }^{\rho }\mathcal{J}_{\eta ,k}^{\alpha ,\beta
}z_{2}\left( x\right) \text{ }^{\rho }\mathcal{J}_{\eta ,k}^{\delta ,\lambda
}z_{1}\left( x\right) \right. ,
\end{eqnarray*}%
which is inequality (\ref{inq66}).
\end{proof}

\begin{corollary}
\label{cor5} Let $z$ be an integrable function on $\left[ 0,\infty \right) $
satisfying $m\leq z\left( x\right) \leq M$, for all $x\in \left[ 0,\infty
\right) $ and $m,M\in
%TCIMACRO{\U{211d} }%
%BeginExpansion
\mathbb{R}
%EndExpansion
$. Then, for all $x>0$ and $\alpha >0,$ $\rho >0,$ $\delta >0,$ $\beta ,\eta
,k,\lambda \in
%TCIMACRO{\U{211d} }%
%BeginExpansion
\mathbb{R}
%EndExpansion
,$ we have%
\begin{eqnarray*}
&&M\Lambda _{x,k}^{\rho ,\beta }\left( \alpha ,\eta \right) \text{ }^{\rho }%
\mathcal{J}_{\eta ,k}^{\delta ,\lambda }v\left( x\right) +m\Lambda
_{x,k}^{\rho ,\lambda }\left( \delta ,\eta \right) \text{ }^{\rho }\mathcal{J%
}_{\eta ,k}^{\alpha ,\beta }v\left( x\right) \\
&&\left. \text{ }\geq \text{ }^{\rho }\mathcal{J}_{\eta ,k}^{\alpha ,\beta
}v\left( x\right) \text{ }^{\rho }\mathcal{J}_{\eta ,k}^{\delta ,\lambda
}v\left( x\right) +mM\Lambda _{x,k}^{\rho ,\beta }\left( \alpha ,\eta
\right) \Lambda _{x,k}^{\rho ,\lambda }\left( \delta ,\eta \right) \right. .
\end{eqnarray*}
\end{corollary}

\begin{remark}
If we put $\eta =0,$ $k=0,$ and taking the limit $\rho \rightarrow 1,$ then
theorem (\ref{thm5}), reduces to theorem 2 and corollary (\ref{cor5}),
reduces to corollary 3 in \cite{ff}.
\end{remark}

Now we give the lemma required for proving our next theorem

\begin{lemma}
\label{LEM1} Let $v,z_{1},z_{2}$ are integrable functions on $\left[
0,\infty \right) $ satisfying the condition (\ref{cond1}) then for all $x>0$
and $\alpha >0,$ $\rho >0,$ $\beta ,\eta ,k\in
%TCIMACRO{\U{211d} }%
%BeginExpansion
\mathbb{R}
%EndExpansion
,$ we have%
\begin{eqnarray}
&&\Lambda _{x,k}^{\rho ,\beta }\left( \alpha ,\eta \right) \text{ }^{\rho }%
\mathcal{J}_{\eta ,k}^{\alpha ,\beta }v^{2}\left( x\right) -\left( \text{ }%
^{\rho }\mathcal{J}_{\eta ,k}^{\alpha ,\beta }v\left( x\right) \right) ^{2}%
\text{ \ \ \ \ \ \ \ \ \ \ \ \ \ \ \ \ \ \ \ \ \ \ \ \ \ \ \ \ \ \ \ \ \ \ \
\ \ \ \ }  \notag \\
&&\left. =\left( \text{ }^{\rho }\mathcal{J}_{\eta ,k}^{\alpha ,\beta
}z_{2}\left( x\right) -\text{ }^{\rho }\mathcal{J}_{\eta ,k}^{\alpha ,\beta
}v\left( x\right) \right) \left( \text{ }^{\rho }\mathcal{J}_{\eta
,k}^{\alpha ,\beta }v\left( x\right) -\text{ }^{\rho }\mathcal{J}_{\eta
,k}^{\alpha ,\beta }z_{1}\left( x\right) \right) \right.  \notag \\
&&\text{ \ \ }-\Lambda _{x,k}^{\rho ,\beta }\left( \alpha ,\eta \right)
\text{ }^{\rho }\mathcal{J}_{\eta ,k}^{\alpha ,\beta }\left[ \left(
z_{2}\left( x\right) -v\left( x\right) \right) \left( v\left( x\right)
-z_{1}\left( x\right) \right) \right]  \notag \\
&&\text{ \ \ }+\Lambda _{x,k}^{\rho ,\beta }\left( \alpha ,\eta \right)
\text{ }^{\rho }\mathcal{J}_{\eta ,k}^{\alpha ,\beta }\left( z_{1}v\right)
\left( x\right) -\text{ }^{\rho }\mathcal{J}_{\eta ,k}^{\alpha ,\beta
}z_{1}\left( x\right) \text{ }^{\rho }\mathcal{J}_{\eta ,k}^{\alpha ,\beta
}v\left( x\right)  \label{id55} \\
&&\text{ \ \ }+\Lambda _{x,k}^{\rho ,\beta }\left( \alpha ,\eta \right)
\text{ }^{\rho }\mathcal{J}_{\eta ,k}^{\alpha ,\beta }\left( z_{2}v\right)
\left( x\right) -\text{ }^{\rho }\mathcal{J}_{\eta ,k}^{\alpha ,\beta
}z_{2}\left( x\right) \text{ }^{\rho }\mathcal{J}_{\eta ,k}^{\alpha ,\beta
}v\left( x\right)  \notag \\
&&\text{ \ \ }-\Lambda _{x,k}^{\rho ,\beta }\left( \alpha ,\eta \right)
\text{ }^{\rho }\mathcal{J}_{\eta ,k}^{\alpha ,\beta }\left(
z_{1}z_{2}\right) \left( x\right) +\text{ }^{\rho }\mathcal{J}_{\eta
,k}^{\alpha ,\beta }z_{2}\left( x\right) \text{ }^{\rho }\mathcal{J}_{\eta
,k}^{\alpha ,\beta }z_{1}\left( x\right) .  \notag
\end{eqnarray}
\end{lemma}

\begin{proof}
For any $\tau ,\sigma >0,$ we have%
\begin{eqnarray}
&&\left( z_{2}\left( \sigma \right) -v\left( \sigma \right) \right) \left(
v\left( \tau \right) -z_{1}\left( \tau \right) \right) +\left( z_{2}\left(
\tau \right) -v\left( \tau \right) \right) \left( v\left( \sigma \right)
-z_{1}\left( \sigma \right) \right)  \notag \\
&&-\left( z_{2}\left( \tau \right) -v\left( \tau \right) \right) \left(
v\left( \tau \right) -z_{1}\left( \tau \right) \right) -\left( z_{2}\left(
\sigma \right) -v\left( \sigma \right) \right) \left( v\left( \sigma \right)
-z_{1}\left( \sigma \right) \right)  \notag \\
&&\left. =v^{2}\left( \tau \right) +v^{2}\left( \sigma \right) -2v\left(
\tau \right) v\left( \sigma \right) \right.  \notag \\
&&\text{ \ \ }+z_{2}\left( \sigma \right) v\left( \tau \right) +z_{1}\left(
\tau \right) v\left( \sigma \right) -z_{1}\left( \tau \right) z_{2}\left(
\sigma \right)  \label{id56} \\
&&\text{ \ \ }+z_{2}\left( \tau \right) v\left( \sigma \right) +z_{1}\left(
\sigma \right) v\left( \tau \right) -z_{1}\left( \sigma \right) z_{2}\left(
\tau \right)  \notag \\
&&\text{ \ \ }-z_{2}\left( \tau \right) v\left( \tau \right) +z_{1}\left(
\tau \right) z_{2}\left( \tau \right) -z_{1}\left( \tau \right) v\left( \tau
\right)  \notag \\
&&\text{ \ \ }-z_{2}\left( \sigma \right) v\left( \sigma \right)
+z_{1}\left( \sigma \right) z_{2}\left( \sigma \right) -z_{1}\left( \sigma
\right) v\left( \sigma \right) .  \notag
\end{eqnarray}%
Multiplying both sides of (\ref{id56}) by $\frac{\rho ^{1-\beta }x^{k}}{%
\Gamma \left( \alpha \right) }\frac{\tau ^{\rho \left( \eta +1\right) -1}}{%
\left( x^{\rho }-\tau ^{\rho }\right) ^{1-\alpha }},$ where $\tau \in \left(
0,x\right) $ and integrating over $\left( 0,x\right) $ with respect to the
variable $\tau ,$ we obtain%
\begin{eqnarray}
&&\left( z_{2}\left( \sigma \right) -v\left( \sigma \right) \right) \left(
\text{ }^{\rho }\mathcal{J}_{\eta ,k}^{\alpha ,\beta }v\left( x\right) -%
\text{ }^{\rho }\mathcal{J}_{\eta ,k}^{\alpha ,\beta }z_{1}\left( x\right)
\right)  \notag \\
&&+\left( v\left( \sigma \right) -z_{1}\left( \sigma \right) \right) \left(
\text{ }^{\rho }\mathcal{J}_{\eta ,k}^{\alpha ,\beta }z_{2}\left( x\right) -%
\text{ }^{\rho }\mathcal{J}_{\eta ,k}^{\alpha ,\beta }v\left( x\right)
\right)  \notag \\
&&-\text{ }^{\rho }\mathcal{J}_{\eta ,k}^{\alpha ,\beta }\left[ \left(
z_{2}\left( \tau \right) -v\left( \tau \right) \right) \left( v\left( \tau
\right) -z_{1}\left( \tau \right) \right) \right]  \notag \\
&&-\left[ \left( z_{2}\left( \sigma \right) -v\left( \sigma \right) \right)
\left( v\left( \sigma \right) -z_{1}\left( \sigma \right) \right) \right]
\Lambda _{x,k}^{\rho ,\beta }\left( \alpha ,\eta \right) \text{ \ \ \ \ }
\label{iden1} \\
&&\left. =\text{ }^{\rho }\mathcal{J}_{\eta ,k}^{\alpha ,\beta }v^{2}\left(
x\right) +v^{2}\left( \sigma \right) -2v\left( \sigma \right) \text{ }^{\rho
}\mathcal{J}_{\eta ,k}^{\alpha ,\beta }v\left( x\right) \right.  \notag \\
&&\text{ \ \ }+z_{2}\left( \sigma \right) \text{ }^{\rho }\mathcal{J}_{\eta
,k}^{\alpha ,\beta }v\left( x\right) +v\left( \sigma \right) \text{ }^{\rho }%
\mathcal{J}_{\eta ,k}^{\alpha ,\beta }z_{1}\left( x\right) -z_{2}\left(
\sigma \right) \text{ }^{\rho }\mathcal{J}_{\eta ,k}^{\alpha ,\beta
}z_{1}\left( x\right)  \notag \\
&&\text{ \ \ }+v\left( \sigma \right) \text{ }^{\rho }\mathcal{J}_{\eta
,k}^{\alpha ,\beta }z_{2}\left( x\right) +z_{1}\left( \sigma \right) \text{ }%
^{\rho }\mathcal{J}_{\eta ,k}^{\alpha ,\beta }v\left( x\right) -z_{1}\left(
\sigma \right) \text{ }^{\rho }\mathcal{J}_{\eta ,k}^{\alpha ,\beta
}z_{2}\left( x\right)  \notag \\
&&\text{ \ \ }-\text{ }^{\rho }\mathcal{J}_{\eta ,k}^{\alpha ,\beta }\left(
z_{2}v\right) \left( x\right) +\text{ }^{\rho }\mathcal{J}_{\eta ,k}^{\alpha
,\beta }\left( z_{1}z_{2}\right) \left( x\right) -\text{ }^{\rho }\mathcal{J}%
_{\eta ,k}^{\alpha ,\beta }\left( z_{1}v\right) \left( x\right)  \notag \\
&&\text{ \ \ }-\Lambda _{x,k}^{\rho ,\beta }\left( \alpha ,\eta \right)
z_{2}\left( \sigma \right) v\left( \sigma \right) +\Lambda _{x,k}^{\rho
,\beta }\left( \alpha ,\eta \right) z_{1}\left( \sigma \right) z_{2}\left(
\sigma \right) -\Lambda _{x,k}^{\rho ,\beta }\left( \alpha ,\eta \right)
z_{1}\left( \sigma \right) v\left( \sigma \right) .  \notag
\end{eqnarray}%
Now multiplying both sides of (\ref{iden1}) by $\frac{\rho ^{1-\beta }x^{k}}{%
\Gamma \left( \alpha \right) }\frac{\sigma ^{\rho \left( \eta +1\right) -1}}{%
\left( x^{\rho }-\sigma ^{\rho }\right) ^{1-\alpha }},$ where $\sigma \in
\left( 0,x\right) $ and integrating over $\left( 0,x\right) $ with respect
to the variable $\sigma ,$ we obtain%
\begin{eqnarray*}
&&\left( \text{ }^{\rho }\mathcal{J}_{\eta ,k}^{\alpha ,\beta }z_{2}\left(
x\right) -\text{ }^{\rho }\mathcal{J}_{\eta ,k}^{\alpha ,\beta }v\left(
x\right) \right) \left( \text{ }^{\rho }\mathcal{J}_{\eta ,k}^{\alpha ,\beta
}v\left( x\right) -\text{ }^{\rho }\mathcal{J}_{\eta ,k}^{\alpha ,\beta
}z_{1}\left( x\right) \right) \\
&&+\left( \text{ }^{\rho }\mathcal{J}_{\eta ,k}^{\alpha ,\beta }v\left(
x\right) -\text{ }^{\rho }\mathcal{J}_{\eta ,k}^{\alpha ,\beta }z_{1}\left(
x\right) \right) \left( \text{ }^{\rho }\mathcal{J}_{\eta ,k}^{\alpha ,\beta
}z_{2}\left( x\right) -\text{ }^{\rho }\mathcal{J}_{\eta ,k}^{\alpha ,\beta
}v\left( x\right) \right) \text{ \ \ \ \ \ \ \ \ \ \ \ \ \ \ \ \ \ \ \ \ \ \
\ \ \ \ \ \ \ \ \ \ \ \ \ \ \ \ } \\
&&-\text{ }^{\rho }\mathcal{J}_{\eta ,k}^{\alpha ,\beta }\left[ \left(
z_{2}\left( x\right) -v\left( x\right) \right) \left( v\left( x\right)
-z_{1}\left( x\right) \right) \right] \Lambda _{x,k}^{\rho ,\beta }\left(
\alpha ,\eta \right) \\
&&-\text{ }^{\rho }\mathcal{J}_{\eta ,k}^{\alpha ,\beta }\left[ \left(
z_{2}\left( x\right) -v\left( x\right) \right) \left( v\left( x\right)
-z_{1}\left( x\right) \right) \right] \Lambda _{x,k}^{\rho ,\beta }\left(
\alpha ,\eta \right) \\
&&\left. =\Lambda _{x,k}^{\rho ,\beta }\left( \alpha ,\eta \right) \text{ }%
^{\rho }\mathcal{J}_{\eta ,k}^{\alpha ,\beta }v^{2}\left( x\right) +\Lambda
_{x,k}^{\rho ,\beta }\left( \alpha ,\eta \right) \text{ }^{\rho }\mathcal{J}%
_{\eta ,k}^{\alpha ,\beta }v^{2}\left( x\right) \right. \\
&&\text{ \ \ }-2\text{ }^{\rho }\mathcal{J}_{\eta ,k}^{\alpha ,\beta
}v\left( x\right) \text{ }^{\rho }\mathcal{J}_{\eta ,k}^{\alpha ,\beta
}v\left( x\right) +\text{ }^{\rho }\mathcal{J}_{\eta ,k}^{\alpha ,\beta
}z_{2}\left( x\right) \text{ }^{\rho }\mathcal{J}_{\eta ,k}^{\alpha ,\beta
}v\left( x\right) \\
&&\text{ \ \ }+\text{ }^{\rho }\mathcal{J}_{\eta ,k}^{\alpha ,\beta }v\left(
x\right) \text{ }^{\rho }\mathcal{J}_{\eta ,k}^{\alpha ,\beta }z_{1}\left(
x\right) -\text{ }^{\rho }\mathcal{J}_{\eta ,k}^{\alpha ,\beta }z_{2}\left(
x\right) \text{ }^{\rho }\mathcal{J}_{\eta ,k}^{\alpha ,\beta }z_{1}\left(
x\right) \\
&&\text{ \ \ }+\text{ }^{\rho }\mathcal{J}_{\eta ,k}^{\alpha ,\beta }v\left(
x\right) \text{ }^{\rho }\mathcal{J}_{\eta ,k}^{\alpha ,\beta }z_{2}\left(
x\right) +\text{ }^{\rho }\mathcal{J}_{\eta ,k}^{\alpha ,\beta }z_{1}\left(
x\right) \text{ }^{\rho }\mathcal{J}_{\eta ,k}^{\alpha ,\beta }v\left(
x\right) \\
&&\text{ \ \ }-\text{ }^{\rho }\mathcal{J}_{\eta ,k}^{\alpha ,\beta
}z_{1}\left( x\right) \text{ }^{\rho }\mathcal{J}_{\eta ,k}^{\alpha ,\beta
}z_{2}\left( x\right) -\Lambda _{x,k}^{\rho ,\beta }\left( \alpha ,\eta
\right) \text{ }^{\rho }\mathcal{J}_{\eta ,k}^{\alpha ,\beta }\left(
z_{2}v\right) \left( x\right) \\
&&\text{ \ \ }+\Lambda _{x,k}^{\rho ,\beta }\left( \alpha ,\eta \right)
\text{ }^{\rho }\mathcal{J}_{\eta ,k}^{\alpha ,\beta }\left(
z_{1}z_{2}\right) \left( x\right) -\Lambda _{x,k}^{\rho ,\beta }\left(
\alpha ,\eta \right) \text{ }^{\rho }\mathcal{J}_{\eta ,k}^{\alpha ,\beta
}\left( z_{1}v\right) \left( x\right) \\
&&\text{ \ \ }-\Lambda _{x,k}^{\rho ,\beta }\left( \alpha ,\eta \right)
\text{ }^{\rho }\mathcal{J}_{\eta ,k}^{\alpha ,\beta }\left( z_{2}v\right)
\left( x\right) +\Lambda _{x,k}^{\rho ,\beta }\left( \alpha ,\eta \right)
\text{ }^{\rho }\mathcal{J}_{\eta ,k}^{\alpha ,\beta }\left(
z_{1}z_{2}\right) \left( x\right) \\
&&\text{ \ \ }-\Lambda _{x,k}^{\rho ,\beta }\left( \alpha ,\eta \right)
\text{ }^{\rho }\mathcal{J}_{\eta ,k}^{\alpha ,\beta }\left( z_{1}v\right)
\left( x\right) .
\end{eqnarray*}%
Which yields the required identity (\ref{id55}).
\end{proof}

Our next result is on Gruss type inequality in case of functional bounds
with same parameters

\begin{theorem}
\label{thm2} Let $v,u$ be two integrable functions on $\left[ 0,\infty
\right) .$ Suppose $z_{1},z_{2},\gamma _{1}$ and $\gamma _{2}$ be four
integrable functions on $\left[ 0,\infty \right) $ satisfying the condition
\begin{equation}
z_{1}\left( x\right) \leq v\left( x\right) \leq z_{2}\left( x\right) \ \ \ \
\ and\ \ \ \ \gamma _{1}\left( x\right) \leq u\left( x\right) \leq \gamma
_{2}\left( x\right) \ \ \ \ \ \forall x\in \left[ 0,\infty \right) .
\label{cond2}
\end{equation}%
Then for all $x>0$ and $\alpha >0,$ $\rho >0,$ $\beta ,\eta ,k\in
%TCIMACRO{\U{211d} }%
%BeginExpansion
\mathbb{R}
%EndExpansion
,$ we have%
\begin{eqnarray}
&&\left[ \Lambda _{x,k}^{\rho ,\beta }\left( \alpha ,\eta \right) \text{ }%
^{\rho }\mathcal{J}_{\eta ,k}^{\alpha ,\beta }\left( vu\right) \left(
x\right) -\left( \text{ }^{\rho }\mathcal{J}_{\eta ,k}^{\alpha ,\beta
}v\left( x\right) \text{ }^{\rho }\mathcal{J}_{\eta ,k}^{\alpha ,\beta
}u\left( x\right) \right) \right] ^{2}  \notag \\
&&\left. \text{ }\leq T\left( v,z_{1},z_{2}\right) T\left( u,\gamma
_{1},\gamma _{2}\right) ,\right.  \label{inq8}
\end{eqnarray}%
where $T\left( \varphi ,\psi ,\omega \right) $ as in \cite{ff}, is defined by%
\begin{eqnarray*}
T\left( \varphi ,\psi ,\omega \right) &=&\left( \text{ }^{\rho }\mathcal{J}%
_{\eta ,k}^{\alpha ,\beta }\omega \left( x\right) -\text{ }^{\rho }\mathcal{J%
}_{\eta ,k}^{\alpha ,\beta }\varphi \left( x\right) \right) \left( \text{ }%
^{\rho }\mathcal{J}_{\eta ,k}^{\alpha ,\beta }\varphi \left( x\right) -\text{
}^{\rho }\mathcal{J}_{\eta ,k}^{\alpha ,\beta }\psi \left( x\right) \right)
\\
&&+\Lambda _{x,k}^{\rho ,\beta }\left( \alpha ,\eta \right) \text{ }^{\rho }%
\mathcal{J}_{\eta ,k}^{\alpha ,\beta }\left( \varphi \psi \right) \left(
x\right) -\text{ }^{\rho }\mathcal{J}_{\eta ,k}^{\alpha ,\beta }\varphi
\left( x\right) \text{ }^{\rho }\mathcal{J}_{\eta ,k}^{\alpha ,\beta }\psi
\left( x\right) \\
&&+\Lambda _{x,k}^{\rho ,\beta }\left( \alpha ,\eta \right) \text{ }^{\rho }%
\mathcal{J}_{\eta ,k}^{\alpha ,\beta }\left( \varphi \omega \right) \left(
x\right) -\text{ }^{\rho }\mathcal{J}_{\eta ,k}^{\alpha ,\beta }\varphi
\left( x\right) \text{ }^{\rho }\mathcal{J}_{\eta ,k}^{\alpha ,\beta }\omega
\left( x\right) \\
&&-\Lambda _{x,k}^{\rho ,\beta }\left( \alpha ,\eta \right) \text{ }^{\rho }%
\mathcal{J}_{\eta ,k}^{\alpha ,\beta }\left( \psi \omega \right) \left(
x\right) +\text{ }^{\rho }\mathcal{J}_{\eta ,k}^{\alpha ,\beta }\psi \left(
x\right) \text{ }^{\rho }\mathcal{J}_{\eta ,k}^{\alpha ,\beta }\omega \left(
x\right) .
\end{eqnarray*}
\end{theorem}

\begin{proof}
Define
\begin{equation}
H\left( \tau ,\sigma \right) :=\left( \ v\left( \tau \right) -v\left( \sigma
\right) \right) \left( u\left( \tau \right) -u\left( \sigma \right) \right) ,%
\text{ \ \ \ \ \ }\tau ,\sigma \in \left( 0,x\right) ,x>0.  \label{cond3}
\end{equation}%
Multiplying both sides of (\ref{cond3}) by $\frac{\rho ^{1-\beta }x^{k}}{%
\Gamma \left( \alpha \right) }\frac{\tau ^{\rho \left( \eta +1\right) -1}}{%
\left( x^{\rho }-\tau ^{\rho }\right) ^{1-\alpha }},$ where $\tau \in \left(
0,x\right) $ and integrating over $\left( 0,x\right) $ with respect to the
variable $\tau ,$ we obtain%
\begin{eqnarray}
&&\frac{\rho ^{1-\beta }x^{k}}{\Gamma \left( \alpha \right) }\int_{0}^{x}%
\frac{\tau ^{\rho \left( \eta +1\right) -1}}{\left( x^{\rho }-\tau ^{\rho
}\right) ^{1-\alpha }}H\left( \tau ,\sigma \right) d\tau  \notag \\
&&\left. :=\text{ }^{\rho }\mathcal{J}_{\eta ,k}^{\alpha ,\beta }\left(
uv\right) \left( x\right) +\Lambda _{x,k}^{\rho ,\beta }\left( \alpha ,\eta
\right) v\left( \sigma \right) u\left( \sigma \right) \right.  \label{id5} \\
&&\text{ \ \ \ \ }-u\left( \sigma \right) \text{ }^{\rho }\mathcal{J}_{\eta
,k}^{\alpha ,\beta }v\left( x\right) -v\left( \sigma \right) \text{ }^{\rho }%
\mathcal{J}_{\eta ,k}^{\alpha ,\beta }u\left( x\right) .  \notag
\end{eqnarray}%
Now multiplying both sides of (\ref{id5}) by $\frac{\rho ^{1-\beta }x^{k}}{%
\Gamma \left( \alpha \right) }\frac{\sigma ^{\rho \left( \eta +1\right) -1}}{%
\left( x^{\rho }-\sigma ^{\rho }\right) ^{1-\alpha }},$ where $\sigma \in
\left( 0,x\right) $ and integrating the resulting identity over $\left(
0,x\right) $ with respect to the variable $\sigma ,$ we get%
\begin{eqnarray}
&&\frac{\rho ^{2\left( 1-\beta \right) }x^{2k}}{2\Gamma ^{2}\left( \alpha
\right) }\int_{0}^{x}\int_{0}^{x}\frac{\tau ^{\rho \left( \eta +1\right) -1}%
}{\left( x^{\rho }-\tau ^{\rho }\right) ^{1-\alpha }}\frac{\sigma ^{\rho
\left( \eta +1\right) -1}}{\left( x^{\rho }-\sigma ^{\rho }\right)
^{1-\alpha }}H\left( \tau ,\sigma \right) d\tau d\sigma  \notag \\
&&\left. :=\Lambda _{x,k}^{\rho ,\beta }\left( \alpha ,\eta \right) \text{ }%
^{\rho }\mathcal{J}_{\eta ,k}^{\alpha ,\beta }\left( uv\right) \left(
x\right) -\text{ }^{\rho }\mathcal{J}_{\eta ,k}^{\alpha ,\beta }u\left(
x\right) \text{ }^{\rho }\mathcal{J}_{\eta ,k}^{\alpha ,\beta }v\left(
x\right) \right. .  \label{id6}
\end{eqnarray}%
Applying the Cauchy-Schwarz inequality to (\ref{id6}), we can write%
\begin{eqnarray}
&&\left( \Lambda _{x,k}^{\rho ,\beta }\left( \alpha ,\eta \right) \text{ }%
^{\rho }\mathcal{J}_{\eta ,k}^{\alpha ,\beta }\left( uv\right) \left(
x\right) -\text{ }^{\rho }\mathcal{J}_{\eta ,k}^{\alpha ,\beta }u\left(
x\right) \text{ }^{\rho }\mathcal{J}_{\eta ,k}^{\alpha ,\beta }v\left(
x\right) \right) ^{2}  \notag \\
&&\left. \text{ }\leq \left( \Lambda _{x,k}^{\rho ,\beta }\left( \alpha
,\eta \right) \text{ }^{\rho }\mathcal{J}_{\eta ,k}^{\alpha ,\beta
}u^{2}\left( x\right) -\left( \text{ }^{\rho }\mathcal{J}_{\eta ,k}^{\alpha
,\beta }u\left( x\right) \right) ^{2}\right) \right.  \label{inq} \\
&&\text{ \ \ }\times \left( \Lambda _{x,k}^{\rho ,\beta }\left( \alpha ,\eta
\right) \text{ }^{\rho }\mathcal{J}_{\eta ,k}^{\alpha ,\beta }v^{2}\left(
x\right) -\left( \text{ }^{\rho }\mathcal{J}_{\eta ,k}^{\alpha ,\beta
}v\left( x\right) \right) ^{2}\right) .  \notag
\end{eqnarray}%
Since
\begin{eqnarray}
\left( z_{2}\left( x\right) -v\left( x\right) \right) \left( v\left(
x\right) -z_{1}\left( x\right) \right) &\geq &0,\text{\ }  \notag \\
\left( \gamma _{2}\left( x\right) -u\left( x\right) \right) \left( u\left(
x\right) -\gamma _{1}\left( x\right) \right) &\geq &0,  \label{inq3}
\end{eqnarray}%
\ for all $x\in \left[ 0,\infty \right) ,$ we have%
\begin{equation*}
\Lambda _{x,k}^{\rho ,\beta }\left( \alpha ,\eta \right) \text{ }^{\rho }%
\mathcal{J}_{\eta ,k}^{\alpha ,\beta }\left( z_{2}\left( x\right) -v\left(
x\right) \right) \left( v\left( x\right) -z_{1}\left( x\right) \right) \geq 0
\end{equation*}%
and%
\begin{equation*}
\Lambda _{x,k}^{\rho ,\beta }\left( \alpha ,\eta \right) \text{ }^{\rho }%
\mathcal{J}_{\eta ,k}^{\alpha ,\beta }\left( \gamma _{2}\left( x\right)
-u\left( x\right) \right) \left( u\left( x\right) -\gamma _{1}\left(
x\right) \right) \geq 0.
\end{equation*}%
Thus, from lemma (\ref{LEM1}), we have%
\begin{eqnarray}
&&\Lambda _{x,k}^{\rho ,\beta }\left( \alpha ,\eta \right) \text{ }^{\rho }%
\mathcal{J}_{\eta ,k}^{\alpha ,\beta }v^{2}\left( x\right) -\left( \text{ }%
^{\rho }\mathcal{J}_{\eta ,k}^{\alpha ,\beta }v\left( x\right) \right) ^{2}
\notag \\
&&\left. \text{ }\leq \left( \text{ }^{\rho }\mathcal{J}_{\eta ,k}^{\alpha
,\beta }z_{2}\left( x\right) -\text{ }^{\rho }\mathcal{J}_{\eta ,k}^{\alpha
,\beta }v\left( x\right) \right) \left( \text{ }^{\rho }\mathcal{J}_{\eta
,k}^{\alpha ,\beta }v\left( x\right) -\text{ }^{\rho }\mathcal{J}_{\eta
,k}^{\alpha ,\beta }z_{1}\left( x\right) \right) \right.  \notag \\
&&\text{ \ \ }+\Lambda _{x,k}^{\rho ,\beta }\left( \alpha ,\eta \right)
\text{ }^{\rho }\mathcal{J}_{\eta ,k}^{\alpha ,\beta }\left( z_{1}v\right)
\left( x\right) -\text{ }^{\rho }\mathcal{J}_{\eta ,k}^{\alpha ,\beta
}z_{1}\left( x\right) \text{ }^{\rho }\mathcal{J}_{\eta ,k}^{\alpha ,\beta
}v\left( x\right)  \label{inq6} \\
&&\text{ \ \ }+\Lambda _{x,k}^{\rho ,\beta }\left( \alpha ,\eta \right)
\text{ }^{\rho }\mathcal{J}_{\eta ,k}^{\alpha ,\beta }\left( z_{2}v\right)
\left( x\right) -\text{ }^{\rho }\mathcal{J}_{\eta ,k}^{\alpha ,\beta
}z_{2}\left( x\right) \text{ }^{\rho }\mathcal{J}_{\eta ,k}^{\alpha ,\beta
}v\left( x\right)  \notag \\
&&\text{ \ \ }-\Lambda _{x,k}^{\rho ,\beta }\left( \alpha ,\eta \right)
\text{ }^{\rho }\mathcal{J}_{\eta ,k}^{\alpha ,\beta }\left(
z_{1}z_{2}\right) \left( x\right) +\text{ }^{\rho }\mathcal{J}_{\eta
,k}^{\alpha ,\beta }z_{2}\left( x\right) \text{ }^{\rho }\mathcal{J}_{\eta
,k}^{\alpha ,\beta }z_{1}\left( x\right)  \notag \\
&&\left. \text{ \ \ \ }=T\left( v,z_{1},z_{2}\right) \right.  \notag
\end{eqnarray}%
and%
\begin{eqnarray}
&&\Lambda _{x,k}^{\rho ,\beta }\left( \alpha ,\eta \right) \text{ }^{\rho }%
\mathcal{J}_{\eta ,k}^{\alpha ,\beta }u^{2}\left( x\right) -\left( \text{ }%
^{\rho }\mathcal{J}_{\eta ,k}^{\alpha ,\beta }u\left( x\right) \right) ^{2}
\notag \\
&&\left. \text{ }\leq \left( \text{ }^{\rho }\mathcal{J}_{\eta ,k}^{\alpha
,\beta }\gamma _{2}\left( x\right) -\text{ }^{\rho }\mathcal{J}_{\eta
,k}^{\alpha ,\beta }u\left( x\right) \right) \left( \text{ }^{\rho }\mathcal{%
J}_{\eta ,k}^{\alpha ,\beta }u\left( x\right) -\text{ }^{\rho }\mathcal{J}%
_{\eta ,k}^{\alpha ,\beta }\gamma _{1}\left( x\right) \right) \right.  \notag
\\
&&\text{ \ \ }+\Lambda _{x,k}^{\rho ,\beta }\left( \alpha ,\eta \right)
\text{ }^{\rho }\mathcal{J}_{\eta ,k}^{\alpha ,\beta }\left( \gamma
_{1}u\right) \left( x\right) -\text{ }^{\rho }\mathcal{J}_{\eta ,k}^{\alpha
,\beta }\gamma _{1}\left( x\right) \text{ }^{\rho }\mathcal{J}_{\eta
,k}^{\alpha ,\beta }u\left( x\right)  \label{inq7} \\
&&\text{ \ \ }+\Lambda _{x,k}^{\rho ,\beta }\left( \alpha ,\eta \right)
\text{ }^{\rho }\mathcal{J}_{\eta ,k}^{\alpha ,\beta }\left( \gamma
_{2}u\right) \left( x\right) -\text{ }^{\rho }\mathcal{J}_{\eta ,k}^{\alpha
,\beta }\gamma _{2}\left( x\right) \text{ }^{\rho }\mathcal{J}_{\eta
,k}^{\alpha ,\beta }u\left( x\right)  \notag \\
&&\text{ \ \ }-\Lambda _{x,k}^{\rho ,\beta }\left( \alpha ,\eta \right)
\text{ }^{\rho }\mathcal{J}_{\eta ,k}^{\alpha ,\beta }\left( \gamma
_{1}\gamma _{2}\right) \left( x\right) +\text{ }^{\rho }\mathcal{J}_{\eta
,k}^{\alpha ,\beta }\gamma _{2}\left( x\right) \text{ }^{\rho }\mathcal{J}%
_{\eta ,k}^{\alpha ,\beta }\gamma _{1}\left( x\right)  \notag \\
&&\left. \text{ \ \ \ }=T\left( u,\gamma _{1},\gamma _{2}\right) \right. .
\notag
\end{eqnarray}%
Combining the Inequalities (\ref{inq6}), (\ref{inq7}) with inequality (\ref%
{inq}), we obtain inequality (\ref{inq8}).
\end{proof}

\begin{remark}
If we put $T\left( v,z_{1},z_{2}\right) =T\left( v,m,M\right) $ $\ \ \ $and $%
\ \ \ T\left( u,\gamma _{1},\gamma _{2}\right) =T\left( v,p,P\right) ,$ in
theorem (\ref{thm2}), where $m,M,p,P$ are constants, then inequality (\ref%
{inq8}) reduces to
\begin{eqnarray*}
&&\left\vert \Lambda _{x,k}^{\rho ,\beta }\left( \alpha ,\eta \right) \text{
}^{\rho }\mathcal{J}_{\eta ,k}^{\alpha ,\beta }\left( vu\right) \left(
x\right) -\left( \text{ }^{\rho }\mathcal{J}_{\eta ,k}^{\alpha ,\beta
}v\left( x\right) \text{ }^{\rho }\mathcal{J}_{\eta ,k}^{\alpha ,\beta
}u\left( x\right) \right) \right\vert \\
&&\left. \text{ }\leq \left( \Lambda _{x,k}^{\rho ,\beta }\left( \alpha
,\eta \right) \right) ^{2}\left( M-m\right) \left( P-p\right) \right. .
\end{eqnarray*}%
Which is result given in \cite{gg}.
\end{remark}

\begin{lemma}
\label{lem2} Let $v,z_{1},z_{2}$ are integrable functions on $\left[
0,\infty \right) $ satisfying the condition (\ref{cond1}), then for all $x>0$
and $\alpha >0,$ $\delta >0,$ $\rho >0,$ $\beta ,\lambda ,\eta ,k\in
%TCIMACRO{\U{211d} }%
%BeginExpansion
\mathbb{R}
%EndExpansion
,$ we have%
\begin{eqnarray}
&&\Lambda _{x,k}^{\rho ,\lambda }\left( \delta ,\eta \right) \text{ }^{\rho }%
\mathcal{J}_{\eta ,k}^{\alpha ,\beta }v^{2}\left( x\right) +\text{ }^{\rho }%
\mathcal{J}_{\eta ,k}^{\delta ,\lambda }v^{2}\left( x\right) -2\text{ }%
^{\rho }\mathcal{J}_{\eta ,k}^{\delta ,\lambda }v\left( x\right) \text{ }%
^{\rho }\mathcal{J}_{\eta ,k}^{\alpha ,\beta }v\left( x\right)  \notag \\
&&\left. =\left( \text{ }^{\rho }\mathcal{J}_{\eta ,k}^{\delta ,\lambda
}z_{2}\left( x\right) -\text{ }^{\rho }\mathcal{J}_{\eta ,k}^{\delta
,\lambda }v\left( x\right) \right) \left( \text{ }^{\rho }\mathcal{J}_{\eta
,k}^{\alpha ,\beta }v\left( x\right) -\text{ }^{\rho }\mathcal{J}_{\eta
,k}^{\alpha ,\beta }z_{1}\left( x\right) \right) \right.  \notag \\
&&\text{ \ \ \ }+\left( \text{ }^{\rho }\mathcal{J}_{\eta ,k}^{\delta
,\lambda }v\left( x\right) -\text{ }^{\rho }\mathcal{J}_{\eta ,k}^{\delta
,\lambda }z_{1}\left( x\right) \right) \left( \text{ }^{\rho }\mathcal{J}%
_{\eta ,k}^{\alpha ,\beta }z_{2}\left( x\right) -\text{ }^{\rho }\mathcal{J}%
_{\eta ,k}^{\alpha ,\beta }v\left( x\right) \right)  \notag \\
&&\text{ \ \ \ }-\text{ }^{\rho }\mathcal{J}_{\eta ,k}^{\alpha ,\beta }\left[
\left( z_{2}\left( \tau \right) -v\left( \tau \right) \right) \left( v\left(
\tau \right) -z_{1}\left( \tau \right) \right) \right] \Lambda _{x,k}^{\rho
,\lambda }\left( \delta ,\eta \right)  \notag \\
&&\text{ \ \ \ }-\text{ }^{\rho }\mathcal{J}_{\eta ,k}^{\delta ,\lambda }%
\left[ \left( z_{2}\left( \sigma \right) -v\left( \sigma \right) \right)
\left( v\left( \sigma \right) -z_{1}\left( \sigma \right) \right) \right]
\Lambda _{x,k}^{\rho ,\beta }\left( \alpha ,\eta \right)  \notag \\
&&\text{ \ \ \ }-\text{ }^{\rho }\mathcal{J}_{\eta ,k}^{\delta ,\lambda
}z_{2}\left( x\right) \text{ }^{\rho }\mathcal{J}_{\eta ,k}^{\alpha ,\beta
}v\left( x\right) -\text{ }^{\rho }\mathcal{J}_{\eta ,k}^{\delta ,\lambda
}v\left( x\right) \text{ }^{\rho }\mathcal{J}_{\eta ,k}^{\alpha ,\beta
}z_{2}\left( x\right)  \label{id7} \\
&&\text{ \ \ \ }-\text{ }^{\rho }\mathcal{J}_{\eta ,k}^{\delta ,\lambda
}v\left( x\right) \text{ }^{\rho }\mathcal{J}_{\eta ,k}^{\alpha ,\beta
}z_{1}\left( x\right) -\text{ }^{\rho }\mathcal{J}_{\eta ,k}^{\delta
,\lambda }z_{1}\left( x\right) \text{ }^{\rho }\mathcal{J}_{\eta ,k}^{\alpha
,\beta }v\left( x\right)  \notag \\
&&\text{ \ \ \ }+\text{ }^{\rho }\mathcal{J}_{\eta ,k}^{\delta ,\lambda
}z_{2}\left( x\right) \text{ }^{\rho }\mathcal{J}_{\eta ,k}^{\alpha ,\beta
}z_{1}\left( x\right) +\text{ }^{\rho }\mathcal{J}_{\eta ,k}^{\delta
,\lambda }z_{1}\left( x\right) \text{ }^{\rho }\mathcal{J}_{\eta ,k}^{\alpha
,\beta }z_{2}\left( x\right)  \notag \\
&&\text{ \ \ \ }+\Lambda _{x,k}^{\rho ,\lambda }\left( \delta ,\eta \right) %
\left[ \text{ }^{\rho }\mathcal{J}_{\eta ,k}^{\alpha ,\beta }\left(
z_{1}v\right) \left( x\right) +\text{ }^{\rho }\mathcal{J}_{\eta ,k}^{\alpha
,\beta }\left( z_{2}v\right) \left( x\right) -\text{ }^{\rho }\mathcal{J}%
_{\eta ,k}^{\alpha ,\beta }\left( z_{1}z_{2}\right) \left( x\right) \right]
\notag \\
&&\text{ \ \ \ }+\Lambda _{x,k}^{\rho ,\beta }\left( \alpha ,\eta \right) %
\left[ \text{ }^{\rho }\mathcal{J}_{\eta ,k}^{\delta ,\lambda }\left(
z_{1}v\right) \left( x\right) +\text{ }^{\rho }\mathcal{J}_{\eta ,k}^{\delta
,\lambda }\left( z_{2}v\right) \left( x\right) -\text{ }^{\rho }\mathcal{J}%
_{\eta ,k}^{\delta ,\lambda }\left( z_{1}z_{2}\right) \left( x\right) \right]
.  \notag
\end{eqnarray}
\end{lemma}

\begin{proof}
In lemma (\ref{LEM1}), multiplying both sides of (\ref{iden1}) by $\frac{%
\rho ^{1-\lambda }x^{k}}{\Gamma \left( \delta \right) }\frac{\sigma ^{\rho
\left( \eta +1\right) -1}}{\left( x^{\rho }-\sigma ^{\rho }\right)
^{1-\delta }},$ where $\sigma \in \left( 0,x\right) $ and integrating the
resulting identity over $\left( 0,x\right) $ with respect to the variable $%
\sigma ,$ we obtain%
\begin{eqnarray*}
&&\left( \text{ }^{\rho }\mathcal{J}_{\eta ,k}^{\delta ,\lambda }z_{2}\left(
x\right) -\text{ }^{\rho }\mathcal{J}_{\eta ,k}^{\delta ,\lambda }v\left(
x\right) \right) \left( \text{ }^{\rho }\mathcal{J}_{\eta ,k}^{\alpha ,\beta
}v\left( x\right) -\text{ }^{\rho }\mathcal{J}_{\eta ,k}^{\alpha ,\beta
}z_{1}\left( x\right) \right) \\
&&+\left( \text{ }^{\rho }\mathcal{J}_{\eta ,k}^{\delta ,\lambda }v\left(
x\right) -\text{ }^{\rho }\mathcal{J}_{\eta ,k}^{\delta ,\lambda
}z_{1}\left( x\right) \right) \left( \text{ }^{\rho }\mathcal{J}_{\eta
,k}^{\alpha ,\beta }z_{2}\left( x\right) -\text{ }^{\rho }\mathcal{J}_{\eta
,k}^{\alpha ,\beta }v\left( x\right) \right) \\
&&-\text{ }^{\rho }\mathcal{J}_{\eta ,k}^{\alpha ,\beta }\left[ \left(
z_{2}\left( \tau \right) -v\left( \tau \right) \right) \left( v\left( \tau
\right) -z_{1}\left( \tau \right) \right) \right] \Lambda _{x,k}^{\rho
,\lambda }\left( \delta ,\eta \right) \\
&&-\text{ }^{\rho }\mathcal{J}_{\eta ,k}^{\delta ,\lambda }\left[ \left(
z_{2}\left( \sigma \right) -v\left( \sigma \right) \right) \left( v\left(
\sigma \right) -z_{1}\left( \sigma \right) \right) \right] \Lambda
_{x,k}^{\rho ,\beta }\left( \alpha ,\eta \right) \\
&&\left. =\Lambda _{x,k}^{\rho ,\lambda }\left( \delta ,\eta \right) \text{ }%
^{\rho }\mathcal{J}_{\eta ,k}^{\alpha ,\beta }v^{2}\left( x\right) +\text{ }%
^{\rho }\mathcal{J}_{\eta ,k}^{\delta ,\lambda }v^{2}\left( x\right) -2\text{
}^{\rho }\mathcal{J}_{\eta ,k}^{\delta ,\lambda }v\left( x\right) \text{ }%
^{\rho }\mathcal{J}_{\eta ,k}^{\alpha ,\beta }v\left( x\right) \right. \\
&&\text{ \ \ \ }+\text{ }^{\rho }\mathcal{J}_{\eta ,k}^{\delta ,\lambda
}z_{2}\left( x\right) \text{ }^{\rho }\mathcal{J}_{\eta ,k}^{\alpha ,\beta
}v\left( x\right) +\text{ }^{\rho }\mathcal{J}_{\eta ,k}^{\delta ,\lambda
}v\left( x\right) \text{ }^{\rho }\mathcal{J}_{\eta ,k}^{\alpha ,\beta
}z_{1}\left( x\right) \\
&&\text{ \ \ \ }-\text{ }^{\rho }\mathcal{J}_{\eta ,k}^{\delta ,\lambda
}z_{2}\left( x\right) \text{ }^{\rho }\mathcal{J}_{\eta ,k}^{\alpha ,\beta
}z_{1}\left( x\right) +\text{ }^{\rho }\mathcal{J}_{\eta ,k}^{\delta
,\lambda }v\left( x\right) \text{ }^{\rho }\mathcal{J}_{\eta ,k}^{\alpha
,\beta }z_{2}\left( x\right) \\
&&\text{ \ \ \ }+\text{ }^{\rho }\mathcal{J}_{\eta ,k}^{\delta ,\lambda
}z_{1}\left( x\right) \text{ }^{\rho }\mathcal{J}_{\eta ,k}^{\alpha ,\beta
}v\left( x\right) -\text{ }^{\rho }\mathcal{J}_{\eta ,k}^{\delta ,\lambda
}z_{1}\left( x\right) \text{ }^{\rho }\mathcal{J}_{\eta ,k}^{\alpha ,\beta
}z_{2}\left( x\right) \\
&&\text{ \ \ \ }-\Lambda _{x,k}^{\rho ,\lambda }\left( \delta ,\eta \right)
\text{ }^{\rho }\mathcal{J}_{\eta ,k}^{\alpha ,\beta }\left( z_{2}v\right)
\left( x\right) +\Lambda _{x,k}^{\rho ,\lambda }\left( \delta ,\eta \right)
\text{ }^{\rho }\mathcal{J}_{\eta ,k}^{\alpha ,\beta }\left(
z_{1}z_{2}\right) \left( x\right) \\
&&\text{ \ \ \ }-\Lambda _{x,k}^{\rho ,\lambda }\left( \delta ,\eta \right)
\text{ }^{\rho }\mathcal{J}_{\eta ,k}^{\alpha ,\beta }\left( z_{1}v\right)
\left( x\right) -\Lambda _{x,k}^{\rho ,\beta }\left( \alpha ,\eta \right)
\text{ }^{\rho }\mathcal{J}_{\eta ,k}^{\delta ,\lambda }\left( z_{2}v\right)
\left( x\right) \\
&&\text{ \ \ \ }+\Lambda _{x,k}^{\rho ,\beta }\left( \alpha ,\eta \right)
\text{ }^{\rho }\mathcal{J}_{\eta ,k}^{\delta ,\lambda }z_{1}z_{2}\left(
x\right) -\Lambda _{x,k}^{\rho ,\beta }\left( \alpha ,\eta \right) \text{ }%
^{\rho }\mathcal{J}_{\eta ,k}^{\delta ,\lambda }z_{1}v\left( x\right) .
\end{eqnarray*}%
Which gives (\ref{id7}) and proves the lemma.
\end{proof}

In our next theorem we prove the result with different parameters. Here we
use our lemma (\ref{lem2}) to proving the result.

\begin{theorem}
Let $v,u$ be two integrable functions on $\left[ 0,\infty \right) $ and
suppose $z_{1},z_{2},\gamma _{1}$ and $\gamma _{2}$ be four integrable
functions on $\left[ 0,\infty \right) $ satisfying the condition (\ref{cond2}%
), then for all $x>0$ and $\alpha >0,$ $\delta >0,$ $\rho >0,$ $\beta
,\lambda ,\eta ,k\in
%TCIMACRO{\U{211d} }%
%BeginExpansion
\mathbb{R}
%EndExpansion
,$ we have%
\begin{eqnarray}
&&\left\vert \Lambda _{x,k}^{\rho ,\lambda }\left( \delta ,\eta \right)
\text{ }^{\rho }\mathcal{J}_{\eta ,k}^{\alpha ,\beta }\left( uv\right)
\left( x\right) +\Lambda _{x,k}^{\rho ,\beta }\left( \alpha ,\eta \right)
\text{ }^{\rho }\mathcal{J}_{\eta ,k}^{\delta ,\lambda }\left( vu\right)
\left( x\right) \right.  \notag \\
&&\left. -\text{ }^{\rho }\mathcal{J}_{\eta ,k}^{\delta ,\lambda }u\left(
x\right) \text{ }^{\rho }\mathcal{J}_{\eta ,k}^{\alpha ,\beta }v\left(
x\right) -\text{ }^{\rho }\mathcal{J}_{\eta ,k}^{\delta ,\lambda }v\left(
x\right) \text{ }^{\rho }\mathcal{J}_{\eta ,k}^{\alpha ,\beta }u\left(
x\right) \right\vert  \notag \\
&&\left. \text{ }\leq \sqrt{K\left( v,z_{1},z_{2}\right) K\left( u,\gamma
_{1},\gamma _{2}\right) }\right. ,  \label{inq12}
\end{eqnarray}%
where $K\left( \varphi ,\psi ,\omega \right) $ is defined by%
\begin{eqnarray*}
&&K\left( \varphi ,\psi ,\omega \right) \\
&&\left. =\left( \text{ }^{\rho }\mathcal{J}_{\eta ,k}^{\delta ,\lambda
}\omega \left( x\right) -\text{ }^{\rho }\mathcal{J}_{\eta ,k}^{\delta
,\lambda }\varphi \left( x\right) \right) \left( \text{ }^{\rho }\mathcal{J}%
_{\eta ,k}^{\alpha ,\beta }\varphi \left( x\right) -\text{ }^{\rho }\mathcal{%
J}_{\eta ,k}^{\alpha ,\beta }\psi \left( x\right) \right) \right. \\
&&\text{ \ \ }+\left( \text{ }^{\rho }\mathcal{J}_{\eta ,k}^{\delta ,\lambda
}\varphi \left( x\right) -\text{ }^{\rho }\mathcal{J}_{\eta ,k}^{\delta
,\lambda }\psi \left( x\right) \right) \left( \text{ }^{\rho }\mathcal{J}%
_{\eta ,k}^{\alpha ,\beta }\omega \left( x\right) -\text{ }^{\rho }\mathcal{J%
}_{\eta ,k}^{\alpha ,\beta }\varphi \left( x\right) \right) \\
&&\text{ \ \ }-\text{ }^{\rho }\mathcal{J}_{\eta ,k}^{\delta ,\lambda
}\omega \left( x\right) \text{ }^{\rho }\mathcal{J}_{\eta ,k}^{\alpha ,\beta
}\varphi \left( x\right) -\text{ }^{\rho }\mathcal{J}_{\eta ,k}^{\delta
,\lambda }\varphi \left( x\right) \text{ }^{\rho }\mathcal{J}_{\eta
,k}^{\alpha ,\beta }\omega \left( x\right) \\
&&\text{ \ \ }-\text{ }^{\rho }\mathcal{J}_{\eta ,k}^{\delta ,\lambda
}\varphi \left( x\right) \text{ }^{\rho }\mathcal{J}_{\eta ,k}^{\alpha
,\beta }\psi \left( x\right) -\text{ }^{\rho }\mathcal{J}_{\eta ,k}^{\delta
,\lambda }\psi \left( x\right) \text{ }^{\rho }\mathcal{J}_{\eta ,k}^{\alpha
,\beta }\varphi \left( x\right) \\
&&\text{ \ \ }+\text{ }^{\rho }\mathcal{J}_{\eta ,k}^{\delta ,\lambda
}\omega \left( x\right) \text{ }^{\rho }\mathcal{J}_{\eta ,k}^{\alpha ,\beta
}\psi \left( x\right) +\text{ }^{\rho }\mathcal{J}_{\eta ,k}^{\delta
,\lambda }\psi \left( x\right) \text{ }^{\rho }\mathcal{J}_{\eta ,k}^{\alpha
,\beta }\omega \left( x\right) \\
&&\text{ \ \ }+\Lambda _{x,k}^{\rho ,\lambda }\left( \delta ,\eta \right) %
\left[ \text{ }^{\rho }\mathcal{J}_{\eta ,k}^{\alpha ,\beta }\left( \psi
\varphi \right) \left( x\right) +\text{ }^{\rho }\mathcal{J}_{\eta
,k}^{\alpha ,\beta }\left( \omega \varphi \right) \left( x\right) -\text{ }%
^{\rho }\mathcal{J}_{\eta ,k}^{\alpha ,\beta }\left( \psi \omega \right)
\left( x\right) \right] \\
&&\text{ \ \ }+\Lambda _{x,k}^{\rho ,\beta }\left( \alpha ,\eta \right) %
\left[ \text{ }^{\rho }\mathcal{J}_{\eta ,k}^{\delta ,\lambda }\left( \psi
\varphi \right) \left( x\right) +\text{ }^{\rho }\mathcal{J}_{\eta
,k}^{\delta ,\lambda }\left( \omega \varphi \right) \left( x\right) -\text{ }%
^{\rho }\mathcal{J}_{\eta ,k}^{\delta ,\lambda }\left( \psi \omega \right)
\left( x\right) \right] .
\end{eqnarray*}
\end{theorem}

\begin{proof}
Multiplying both sides of (\ref{id5}) by $\frac{\rho ^{1-\lambda }x^{k}}{%
\Gamma \left( \delta \right) }\frac{\sigma ^{\rho \left( \eta +1\right) -1}}{%
\left( x^{\rho }-\sigma ^{\rho }\right) ^{1-\delta }},$ where $\sigma \in
\left( 0,x\right) $ and integrating the resulting identity over $\left(
0,x\right) $ with respect to the variable $\sigma ,$ we obtain%
\begin{eqnarray}
&&\frac{\rho ^{2-\beta -\lambda }x^{2k}}{\Gamma \left( \alpha \right) \Gamma
\left( \delta \right) }\int_{0}^{x}\int_{0}^{x}\frac{\tau ^{\rho \left( \eta
+1\right) -1}}{\left( x^{\rho }-\tau ^{\rho }\right) ^{1-\alpha }}\frac{\rho
^{1-\lambda }x^{k}}{\Gamma \left( \delta \right) }\frac{\sigma ^{\rho \left(
\eta +1\right) -1}}{\left( x^{\rho }-\sigma ^{\rho }\right) ^{1-\delta }}%
H\left( \tau ,\sigma \right) d\tau d\sigma  \notag \\
&&\left. :=\Lambda _{x,k}^{\rho ,\lambda }\left( \delta ,\eta \right) \text{
}^{\rho }\mathcal{J}_{\eta ,k}^{\alpha ,\beta }\left( uv\right) \left(
x\right) +\Lambda _{x,k}^{\rho ,\beta }\left( \alpha ,\eta \right) \text{ }%
^{\rho }\mathcal{J}_{\eta ,k}^{\delta ,\lambda }\left( vu\right) \left(
x\right) \right.  \label{id8} \\
&&\text{ \ \ \ }-\text{ }^{\rho }\mathcal{J}_{\eta ,k}^{\delta ,\lambda
}u\left( x\right) \text{ }^{\rho }\mathcal{J}_{\eta ,k}^{\alpha ,\beta
}v\left( x\right) -\text{ }^{\rho }\mathcal{J}_{\eta ,k}^{\delta ,\lambda
}v\left( x\right) \text{ }^{\rho }\mathcal{J}_{\eta ,k}^{\alpha ,\beta
}u\left( x\right) .  \notag
\end{eqnarray}%
Applying Cauchy-Schearz inequality for double integrals, we get%
\begin{eqnarray}
&&\left[ \Lambda _{x,k}^{\rho ,\lambda }\left( \delta ,\eta \right) \text{ }%
^{\rho }\mathcal{J}_{\eta ,k}^{\alpha ,\beta }\left( uv\right) \left(
x\right) +\Lambda _{x,k}^{\rho ,\beta }\left( \alpha ,\eta \right) \text{ }%
^{\rho }\mathcal{J}_{\eta ,k}^{\delta ,\lambda }vu\left( x\right) \right.
\notag \\
&&\left. -\text{ }^{\rho }\mathcal{J}_{\eta ,k}^{\delta ,\lambda }u\left(
x\right) \text{ }^{\rho }\mathcal{J}_{\eta ,k}^{\alpha ,\beta }v\left(
x\right) -\text{ }^{\rho }\mathcal{J}_{\eta ,k}^{\delta ,\lambda }v\left(
x\right) \text{ }^{\rho }\mathcal{J}_{\eta ,k}^{\alpha ,\beta }u\left(
x\right) \right] ^{2}  \label{inq9} \\
&&\left. \leq \left( \Lambda _{x,k}^{\rho ,\lambda }\left( \delta ,\eta
\right) \text{ }^{\rho }\mathcal{J}_{\eta ,k}^{\alpha ,\beta }v^{2}\left(
x\right) +\Lambda _{x,k}^{\rho ,\beta }\left( \alpha ,\eta \right) \text{ }%
^{\rho }\mathcal{J}_{\eta ,k}^{\delta ,\lambda }v^{2}\left( x\right) -2\text{
}^{\rho }\mathcal{J}_{\eta ,k}^{\delta ,\lambda }v\left( x\right) \text{ }%
^{\rho }\mathcal{J}_{\eta ,k}^{\alpha ,\beta }v\left( x\right) \right)
\right.  \notag \\
&&\text{ }\times \left( \Lambda _{x,k}^{\rho ,\lambda }\left( \delta ,\eta
\right) \text{ }^{\rho }\mathcal{J}_{\eta ,k}^{\alpha ,\beta }u^{2}\left(
x\right) +\Lambda _{x,k}^{\rho ,\beta }\left( \alpha ,\eta \right) \text{ }%
^{\rho }\mathcal{J}_{\eta ,k}^{\delta ,\lambda }u^{2}\left( x\right) -2\text{
}^{\rho }\mathcal{J}_{\eta ,k}^{\delta ,\lambda }u\left( x\right) \text{ }%
^{\rho }\mathcal{J}_{\eta ,k}^{\alpha ,\beta }u\left( x\right) \right) .
\notag
\end{eqnarray}%
Since
\begin{equation*}
\left( z_{2}\left( x\right) -v\left( x\right) \right) \left( v\left(
x\right) -z_{1}\left( x\right) \right) \geq 0\text{ \ \ \ \ }
\end{equation*}%
\ and
\begin{equation*}
\left( \gamma _{2}\left( x\right) -u\left( x\right) \right) \left( u\left(
x\right) -\gamma _{1}\left( x\right) \right) \geq 0,
\end{equation*}%
\ \ \ \ for all $x\in \left[ 0,\infty \right) ,$ therefore%
\begin{eqnarray*}
&&\left( \text{ }^{\rho }\mathcal{J}_{\eta ,k}^{\alpha ,\beta }\left[ \left(
z_{2}\left( \tau \right) -v\left( \tau \right) \right) \left( v\left( \tau
\right) -z_{1}\left( \tau \right) \right) \right] \Lambda _{x,k}^{\rho
,\lambda }\left( \delta ,\eta \right) \right. \\
&&\left. +\text{ }^{\rho }\mathcal{J}_{\eta ,k}^{\delta ,\lambda }\left[
\left( z_{2}\left( \sigma \right) -v\left( \sigma \right) \right) \left(
v\left( \sigma \right) -z_{1}\left( \sigma \right) \right) \right] \Lambda
_{x,k}^{\rho ,\beta }\left( \alpha ,\eta \right) \right) \\
&&\left. \text{ }\geq 0\right. ,
\end{eqnarray*}%
and%
\begin{eqnarray*}
&&\left( \text{ }^{\rho }\mathcal{J}_{\eta ,k}^{\alpha ,\beta }\left[ \left(
\gamma _{2}\left( \tau \right) -u\left( \tau \right) \right) \left( u\left(
\tau \right) -\gamma _{1}\left( \tau \right) \right) \right] \Lambda
_{x,k}^{\rho ,\lambda }\left( \delta ,\eta \right) \right. \text{ \ } \\
&&\left. +\text{ }^{\rho }\mathcal{J}_{\eta ,k}^{\delta ,\lambda }\left[
\left( \gamma _{2}\left( \sigma \right) -u\left( \sigma \right) \right)
\left( u\left( \sigma \right) -\gamma _{1}\left( \sigma \right) \right) %
\right] \Lambda _{x,k}^{\rho ,\beta }\left( \alpha ,\eta \right) \right) \\
&&\left. \text{ }\geq 0\right. .
\end{eqnarray*}%
Thus, from lemma (\ref{lem2}), we have%
\begin{eqnarray}
&&\Lambda _{x,k}^{\rho ,\lambda }\left( \delta ,\eta \right) \text{ }^{\rho }%
\mathcal{J}_{\eta ,k}^{\alpha ,\beta }v^{2}\left( x\right) +\text{ }^{\rho }%
\mathcal{J}_{\eta ,k}^{\delta ,\lambda }v^{2}\left( x\right) -2\text{ }%
^{\rho }\mathcal{J}_{\eta ,k}^{\delta ,\lambda }v\left( x\right) \text{ }%
^{\rho }\mathcal{J}_{\eta ,k}^{\alpha ,\beta }v\left( x\right)  \notag \\
&&\left. \leq \left( \text{ }^{\rho }\mathcal{J}_{\eta ,k}^{\delta ,\lambda
}z_{2}\left( x\right) -\text{ }^{\rho }\mathcal{J}_{\eta ,k}^{\delta
,\lambda }v\left( x\right) \right) \left( \text{ }^{\rho }\mathcal{J}_{\eta
,k}^{\alpha ,\beta }v\left( x\right) -\text{ }^{\rho }\mathcal{J}_{\eta
,k}^{\alpha ,\beta }z_{1}\left( x\right) \right) \right.  \notag \\
&&\text{ }+\left( \text{ }^{\rho }\mathcal{J}_{\eta ,k}^{\delta ,\lambda
}v\left( x\right) -\text{ }^{\rho }\mathcal{J}_{\eta ,k}^{\delta ,\lambda
}z_{1}\left( x\right) \right) \left( \text{ }^{\rho }\mathcal{J}_{\eta
,k}^{\alpha ,\beta }z_{2}\left( x\right) -\text{ }^{\rho }\mathcal{J}_{\eta
,k}^{\alpha ,\beta }v\left( x\right) \right)  \notag \\
&&\text{ }-\text{ }^{\rho }\mathcal{J}_{\eta ,k}^{\delta ,\lambda
}z_{2}\left( x\right) \text{ }^{\rho }\mathcal{J}_{\eta ,k}^{\alpha ,\beta
}v\left( x\right) -\text{ }^{\rho }\mathcal{J}_{\eta ,k}^{\delta ,\lambda
}v\left( x\right) \text{ }^{\rho }\mathcal{J}_{\eta ,k}^{\alpha ,\beta
}z_{2}\left( x\right)  \notag \\
&&\text{ }-\text{ }^{\rho }\mathcal{J}_{\eta ,k}^{\delta ,\lambda }v\left(
x\right) \text{ }^{\rho }\mathcal{J}_{\eta ,k}^{\alpha ,\beta }z_{1}\left(
x\right) -\text{ }^{\rho }\mathcal{J}_{\eta ,k}^{\delta ,\lambda
}z_{1}\left( x\right) \text{ }^{\rho }\mathcal{J}_{\eta ,k}^{\alpha ,\beta
}v\left( x\right)  \label{inq10} \\
&&\text{ }+\text{ }^{\rho }\mathcal{J}_{\eta ,k}^{\delta ,\lambda
}z_{2}\left( x\right) \text{ }^{\rho }\mathcal{J}_{\eta ,k}^{\alpha ,\beta
}z_{1}\left( x\right) +\text{ }^{\rho }\mathcal{J}_{\eta ,k}^{\delta
,\lambda }z_{1}\left( x\right) \text{ }^{\rho }\mathcal{J}_{\eta ,k}^{\alpha
,\beta }z_{2}\left( x\right)  \notag \\
&&\text{ }+\Lambda _{x,k}^{\rho ,\lambda }\left( \delta ,\eta \right) \left[
\text{ }^{\rho }\mathcal{J}_{\eta ,k}^{\alpha ,\beta }\left( z_{1}v\right)
\left( x\right) +\text{ }^{\rho }\mathcal{J}_{\eta ,k}^{\alpha ,\beta
}\left( z_{2}v\right) \left( x\right) -\text{ }^{\rho }\mathcal{J}_{\eta
,k}^{\alpha ,\beta }\left( z_{1}z_{2}\right) \left( x\right) \right]  \notag
\\
&&\text{ }+\Lambda _{x,k}^{\rho ,\beta }\left( \alpha ,\eta \right) \left[
\text{ }^{\rho }\mathcal{J}_{\eta ,k}^{\delta ,\lambda }\left( z_{1}v\right)
\left( x\right) +\text{ }^{\rho }\mathcal{J}_{\eta ,k}^{\delta ,\lambda
}\left( z_{2}v\right) \left( x\right) -\text{ }^{\rho }\mathcal{J}_{\eta
,k}^{\delta ,\lambda }\left( z_{1}z_{2}\right) \left( x\right) \right]
\notag \\
&&\left. \text{ }=K\left( v,z_{1},z_{2}\right) \right. ,  \notag
\end{eqnarray}%
\begin{eqnarray}
&&\Lambda _{x,k}^{\rho ,\lambda }\left( \delta ,\eta \right) \text{ }^{\rho }%
\mathcal{J}_{\eta ,k}^{\alpha ,\beta }u^{2}\left( x\right) +\text{ }^{\rho }%
\mathcal{J}_{\eta ,k}^{\delta ,\lambda }u^{2}\left( x\right) -2\text{ }%
^{\rho }\mathcal{J}_{\eta ,k}^{\delta ,\lambda }u\left( x\right) \text{ }%
^{\rho }\mathcal{J}_{\eta ,k}^{\alpha ,\beta }u\left( x\right)  \notag \\
&&\left. \leq \left( \text{ }^{\rho }\mathcal{J}_{\eta ,k}^{\delta ,\lambda
}\gamma _{2}\left( x\right) -\text{ }^{\rho }\mathcal{J}_{\eta ,k}^{\delta
,\lambda }u\left( x\right) \right) \left( \text{ }^{\rho }\mathcal{J}_{\eta
,k}^{\alpha ,\beta }u\left( x\right) -\text{ }^{\rho }\mathcal{J}_{\eta
,k}^{\alpha ,\beta }\gamma _{1}\left( x\right) \right) \right.  \notag \\
&&\text{ }+\left( \text{ }^{\rho }\mathcal{J}_{\eta ,k}^{\delta ,\lambda
}u\left( x\right) -\text{ }^{\rho }\mathcal{J}_{\eta ,k}^{\delta ,\lambda
}\gamma _{1}\left( x\right) \right) \left( \text{ }^{\rho }\mathcal{J}_{\eta
,k}^{\alpha ,\beta }\gamma _{2}\left( x\right) -\text{ }^{\rho }\mathcal{J}%
_{\eta ,k}^{\alpha ,\beta }u\left( x\right) \right)  \notag \\
&&\text{ }-\text{ }^{\rho }\mathcal{J}_{\eta ,k}^{\delta ,\lambda }\gamma
_{2}\left( x\right) \text{ }^{\rho }\mathcal{J}_{\eta ,k}^{\alpha ,\beta
}u\left( x\right) -\text{ }^{\rho }\mathcal{J}_{\eta ,k}^{\delta ,\lambda
}u\left( x\right) \text{ }^{\rho }\mathcal{J}_{\eta ,k}^{\alpha ,\beta
}\gamma _{2}\left( x\right)  \notag \\
&&\text{ }-\text{ }^{\rho }\mathcal{J}_{\eta ,k}^{\delta ,\lambda }u\left(
x\right) \text{ }^{\rho }\mathcal{J}_{\eta ,k}^{\alpha ,\beta }\gamma
_{1}\left( x\right) -\text{ }^{\rho }\mathcal{J}_{\eta ,k}^{\delta ,\lambda
}\gamma _{1}\left( x\right) \text{ }^{\rho }\mathcal{J}_{\eta ,k}^{\alpha
,\beta }u\left( x\right)  \label{inq11} \\
&&\text{ }+\text{ }^{\rho }\mathcal{J}_{\eta ,k}^{\delta ,\lambda }\gamma
_{2}\left( x\right) \text{ }^{\rho }\mathcal{J}_{\eta ,k}^{\alpha ,\beta
}\gamma _{1}\left( x\right) +\text{ }^{\rho }\mathcal{J}_{\eta ,k}^{\delta
,\lambda }\gamma _{1}\left( x\right) \text{ }^{\rho }\mathcal{J}_{\eta
,k}^{\alpha ,\beta }\gamma _{2}\left( x\right)  \notag \\
&&\text{ }+\Lambda _{x,k}^{\rho ,\lambda }\left( \delta ,\eta \right) \left[
\text{ }^{\rho }\mathcal{J}_{\eta ,k}^{\alpha ,\beta }\left( \gamma
_{1}u\right) \left( x\right) +\text{ }^{\rho }\mathcal{J}_{\eta ,k}^{\alpha
,\beta }\left( \gamma _{2}u\right) \left( x\right) -\text{ }^{\rho }\mathcal{%
J}_{\eta ,k}^{\alpha ,\beta }\left( \gamma _{1}\gamma _{2}\right) \left(
x\right) \right]  \notag \\
&&\text{ }+\Lambda _{x,k}^{\rho ,\beta }\left( \alpha ,\eta \right) \left[
\text{ }^{\rho }\mathcal{J}_{\eta ,k}^{\delta ,\lambda }\left( \gamma
_{1}u\right) \left( x\right) +\text{ }^{\rho }\mathcal{J}_{\eta ,k}^{\delta
,\lambda }\left( \gamma _{2}u\right) \left( x\right) -\text{ }^{\rho }%
\mathcal{J}_{\eta ,k}^{\delta ,\lambda }\left( \gamma _{1}\gamma _{2}\right)
\left( x\right) \right]  \notag \\
&&\left. \text{ }=K\left( u,\gamma _{1},\gamma _{2}\right) \right. .  \notag
\end{eqnarray}%
From the Inequalities (\ref{inq10}), (\ref{inq11}) and inequality (\ref{inq9}%
), we obtain inequality (\ref{inq12}).
\end{proof}

Now we give the following result

\begin{theorem}
\label{thm6} Let $v,u$ be two integrable functions on $\left[ 0,\infty
\right) $ and suppose $z_{1},z_{2},\gamma _{1}$ and $\gamma _{2}$ be four
integrable functions on $\left[ 0,\infty \right) $ satisfying the condition (%
\ref{cond2}), then for all $x>0$ and $\alpha >0,$ $\delta >0,$ $\rho >0,$ $%
\beta ,\lambda ,\eta ,k\in
%TCIMACRO{\U{211d} }%
%BeginExpansion
\mathbb{R}
%EndExpansion
,$ the following inequalities holds:%
\begin{eqnarray*}
&&\left( a\right) \text{ }^{\rho }\mathcal{J}_{\eta ,k}^{\delta ,\lambda
}u\left( x\right) \text{ }^{\rho }\mathcal{J}_{\eta ,k}^{\alpha ,\beta
}z_{2}\left( x\right) +\text{ }^{\rho }\mathcal{J}_{\eta ,k}^{\delta
,\lambda }\gamma _{1}\left( x\right) \text{ }^{\rho }\mathcal{J}_{\eta
,k}^{\alpha ,\beta }v\left( x\right) \\
&&\left. \text{ \ \ \ \ }\geq \text{ }^{\rho }\mathcal{J}_{\eta ,k}^{\delta
,\lambda }\gamma _{1}\left( x\right) \text{ }^{\rho }\mathcal{J}_{\eta
,k}^{\alpha ,\beta }z_{2}\left( x\right) +\text{ }^{\rho }\mathcal{J}_{\eta
,k}^{\delta ,\lambda }u\left( x\right) \text{ }^{\rho }\mathcal{J}_{\eta
,k}^{\alpha ,\beta }v\left( x\right) \right. ,
\end{eqnarray*}
\end{theorem}

\begin{eqnarray*}
&&\left( b\right) \text{ }^{\rho }\mathcal{J}_{\eta ,k}^{\delta ,\lambda
}z_{1}\left( x\right) \text{ }^{\rho }\mathcal{J}_{\eta ,k}^{\alpha ,\beta
}u\left( x\right) +\text{ }^{\rho }\mathcal{J}_{\eta ,k}^{\alpha ,\beta
}\gamma _{2}\left( x\right) \text{ }^{\rho }\mathcal{J}_{\eta ,k}^{\delta
,\lambda }v\left( x\right) \\
&&\left. \text{ \ \ \ \ }\geq \text{ }^{\rho }\mathcal{J}_{\eta ,k}^{\delta
,\lambda }z_{1}\left( x\right) \text{ }^{\rho }\mathcal{J}_{\eta ,k}^{\alpha
,\beta }\gamma _{2}\left( x\right) +\text{ }^{\rho }\mathcal{J}_{\eta
,k}^{\delta ,\lambda }v\left( x\right) \text{ }^{\rho }\mathcal{J}_{\eta
,k}^{\alpha ,\beta }u\left( x\right) \right. ,
\end{eqnarray*}

\begin{eqnarray*}
&&\left( c\right) \text{ }^{\rho }\mathcal{J}_{\eta ,k}^{\alpha ,\beta
}z_{2}\left( x\right) \text{ }^{\rho }\mathcal{J}_{\eta ,k}^{\delta ,\lambda
}\gamma _{2}\left( x\right) +\text{ }^{\rho }\mathcal{J}_{\eta ,k}^{\alpha
,\beta }v\left( x\right) \text{ }^{\rho }\mathcal{J}_{\eta ,k}^{\delta
,\lambda }u\left( x\right) \\
&&\left. \text{ \ \ \ \ }\geq \text{ }^{\rho }\mathcal{J}_{\eta ,k}^{\alpha
,\beta }z_{2}\left( x\right) \text{ }^{\rho }\mathcal{J}_{\eta ,k}^{\delta
,\lambda }u\left( x\right) +\text{ }^{\rho }\mathcal{J}_{\eta ,k}^{\delta
,\lambda }\gamma _{2}\left( x\right) \text{ }^{\rho }\mathcal{J}_{\eta
,k}^{\alpha ,\beta }v\left( x\right) \right. ,
\end{eqnarray*}

\begin{eqnarray*}
&&\left( d\right) \text{ }^{\rho }\mathcal{J}_{\eta ,k}^{\alpha ,\beta
}z_{1}\left( x\right) \text{ }^{\rho }\mathcal{J}_{\eta ,k}^{\delta ,\lambda
}\gamma _{1}\left( x\right) +\text{ }^{\rho }\mathcal{J}_{\eta ,k}^{\alpha
,\beta }v\left( x\right) \text{ }^{\rho }\mathcal{J}_{\eta ,k}^{\delta
,\lambda }u\left( x\right) \\
&&\left. \text{ \ \ \ \ }\geq \text{ }^{\rho }\mathcal{J}_{\eta ,k}^{\alpha
,\beta }z_{1}\left( x\right) \text{ }^{\rho }\mathcal{J}_{\eta ,k}^{\delta
,\lambda }u\left( x\right) +\text{ }^{\rho }\mathcal{J}_{\eta ,k}^{\delta
,\lambda }\gamma _{1}\left( x\right) \text{ }^{\rho }\mathcal{J}_{\eta
,k}^{\alpha ,\beta }v\left( x\right) \right. .
\end{eqnarray*}

\begin{proof}
To prove $\left( a\right) $, from the condition (\ref{cond2}), we have for $%
x\in \left[ 0,\infty \right) $ that
\begin{equation}
\left( z_{2}\left( \tau \right) -v\left( \tau \right) \right) \left( u\left(
\sigma \right) -\gamma _{1}\left( \sigma \right) \right) \geq 0.
\label{inq13}
\end{equation}%
Therefore%
\begin{equation}
z_{2}\left( \tau \right) u\left( \sigma \right) +v\left( \tau \right) \gamma
_{1}\left( \sigma \right) \geq z_{2}\left( \tau \right) \gamma _{1}\left(
\sigma \right) +v\left( \tau \right) u\left( \sigma \right) .  \label{inq14}
\end{equation}%
Multiplying both sides of (\ref{inq14}) by $\frac{\rho ^{1-\beta }x^{k}}{%
\Gamma \left( \alpha \right) }\frac{\tau ^{\rho \left( \eta +1\right) -1}}{%
\left( x^{\rho }-\tau ^{\rho }\right) ^{1-\alpha }},$ where $\tau \in \left(
0,x\right) $ and integrating over $\left( 0,x\right) $ with respect to the
variable $\tau ,$ we obtain%
\begin{eqnarray}
&&u\left( \sigma \right) \text{ }^{\rho }\mathcal{J}_{\eta ,k}^{\alpha
,\beta }z_{2}\left( x\right) +\gamma _{1}\left( \sigma \right) \text{ }%
^{\rho }\mathcal{J}_{\eta ,k}^{\alpha ,\beta }v\left( x\right)  \notag \\
&&\left. \text{ }\geq \gamma _{1}\left( \sigma \right) \text{ }^{\rho }%
\mathcal{J}_{\eta ,k}^{\alpha ,\beta }z_{2}\left( x\right) +u\left( \sigma
\right) \text{ }^{\rho }\mathcal{J}_{\eta ,k}^{\alpha ,\beta }v\left(
x\right) \right. .  \label{inq15}
\end{eqnarray}%
Now multiplying both sides of (\ref{id5}) by $\frac{\rho ^{1-\lambda }x^{k}}{%
\Gamma \left( \delta \right) }\frac{\sigma ^{\rho \left( \eta +1\right) -1}}{%
\left( x^{\rho }-\sigma ^{\rho }\right) ^{1-\delta }},$ where $\sigma \in
\left( 0,x\right) $ and integrating the resulting identity over $\left(
0,x\right) $ with respect to the variable $\sigma ,$ we get the desired
inequality $\left( a\right) .$To prove $\left( b\right) ,\left( c\right) $
and $\left( d\right) ,$ we use the following inequalities:%
\begin{equation*}
\text{ \ }\left( b\right) \text{\ \ \ \ \ \ }\left( \gamma _{2}\left( \tau
\right) -u\left( \tau \right) \right) \left( v\left( \sigma \right)
-z_{1}\left( \sigma \right) \right) \geq 0,
\end{equation*}%
\begin{equation*}
\text{ \ }\left( b\right) \text{\ \ \ \ \ \ }\left( z_{2}\left( \tau \right)
-v\left( \tau \right) \right) \left( u\left( \sigma \right) -\gamma
_{2}\left( \sigma \right) \right) \leq 0,
\end{equation*}%
\begin{equation*}
\text{ \ }\left( b\right) \text{\ \ \ \ \ \ }\left( z_{1}\left( \tau \right)
-v\left( \tau \right) \right) \left( u\left( \sigma \right) -\gamma
_{1}\left( \sigma \right) \right) \leq 0.
\end{equation*}
\end{proof}

The next corollary is a special case of Theorem (\ref{thm6}).

\begin{corollary}
\label{cor7} Let $v,u$ be two integrable functions on $\left[ 0,\infty
\right) $ and suppose that there exist the constants $n,N,m$,$M$ satisfying
the condition
\begin{equation*}
m\leq v\left( x\right) \leq M\ \ \ \ \ and\ \ \ \ n\leq u\left( x\right)
\leq N\ \ \ \ \ \forall x\in \left[ 0,\infty \right) ,
\end{equation*}%
, then for all $x>0$ and $\alpha >0,$ $\delta >0,$ $\rho >0,$ $\beta
,\lambda ,\eta ,k\in
%TCIMACRO{\U{211d} }%
%BeginExpansion
\mathbb{R}
%EndExpansion
,$ we have:%
\begin{eqnarray*}
&&\left( A\right) \text{\ \ }M\Lambda _{x,k}^{\rho ,\beta }\left( \alpha
,\eta \right) \text{ }^{\rho }\mathcal{J}_{\eta ,k}^{\delta ,\lambda
}u\left( x\right) +n\Lambda _{x,k}^{\rho ,\lambda }\left( \delta ,\eta
\right) \text{ }^{\rho }\mathcal{J}_{\eta ,k}^{\alpha ,\beta }v\left(
x\right) \\
&&\left. \text{ \ \ \ \ }\geq nM\Lambda _{x,k}^{\rho ,\lambda }\left( \delta
,\eta \right) \Lambda _{x,k}^{\rho ,\beta }\left( \alpha ,\eta \right) +%
\text{ }^{\rho }\mathcal{J}_{\eta ,k}^{\delta ,\lambda }u\left( x\right)
\text{ }^{\rho }\mathcal{J}_{\eta ,k}^{\alpha ,\beta }v\left( x\right)
\right. ,
\end{eqnarray*}
\end{corollary}

\begin{eqnarray*}
&&\left( B\right) \text{ \ }m\Lambda _{x,k}^{\rho ,\lambda }\left( \delta
,\eta \right) \text{ }^{\rho }\mathcal{J}_{\eta ,k}^{\alpha ,\beta }u\left(
x\right) +N\Lambda _{x,k}^{\rho ,\beta }\left( \alpha ,\eta \right) \text{ }%
^{\rho }\mathcal{J}_{\eta ,k}^{\delta ,\lambda }v\left( x\right) \\
&&\left. \text{ \ \ \ \ }\geq mN\Lambda _{x,k}^{\rho ,\lambda }\left( \delta
,\eta \right) \Lambda _{x,k}^{\rho ,\beta }\left( \alpha ,\eta \right) +%
\text{ }^{\rho }\mathcal{J}_{\eta ,k}^{\delta ,\lambda }v\left( x\right)
\text{ }^{\rho }\mathcal{J}_{\eta ,k}^{\alpha ,\beta }u\left( x\right)
\right. ,
\end{eqnarray*}

\begin{eqnarray*}
&&\left( C\right) \text{ \ }MN\Lambda _{x,k}^{\rho ,\beta }\left( \alpha
,\eta \right) \Lambda _{x,k}^{\rho ,\lambda }\left( \delta ,\eta \right) +%
\text{ }^{\rho }\mathcal{J}_{\eta ,k}^{\alpha ,\beta }v\left( x\right) \text{
}^{\rho }\mathcal{J}_{\eta ,k}^{\delta ,\lambda }u\left( x\right) \\
&&\left. \text{ \ \ \ \ }\geq M\Lambda _{x,k}^{\rho ,\beta }\left( \alpha
,\eta \right) \text{ }^{\rho }\mathcal{J}_{\eta ,k}^{\delta ,\lambda
}u\left( x\right) +N\Lambda _{x,k}^{\rho ,\lambda }\left( \delta ,\eta
\right) \text{ }^{\rho }\mathcal{J}_{\eta ,k}^{\alpha ,\beta }v\left(
x\right) \right. ,
\end{eqnarray*}

\begin{eqnarray*}
&&\left( D\right) \text{ }mn\Lambda _{x,k}^{\rho ,\beta }\left( \alpha ,\eta
\right) \Lambda _{x,k}^{\rho ,\lambda }\left( \delta ,\eta \right) +\text{ }%
^{\rho }\mathcal{J}_{\eta ,k}^{\alpha ,\beta }v\left( x\right) \text{ }%
^{\rho }\mathcal{J}_{\eta ,k}^{\delta ,\lambda }u\left( x\right) \\
&&\left. \text{ \ \ \ \ }\geq m\Lambda _{x,k}^{\rho ,\beta }\left( \alpha
,\eta \right) \text{ }^{\rho }\mathcal{J}_{\eta ,k}^{\delta ,\lambda
}u\left( x\right) +n\Lambda _{x,k}^{\rho ,\lambda }\left( \delta ,\eta
\right) \text{ }^{\rho }\mathcal{J}_{\eta ,k}^{\alpha ,\beta }v\left(
x\right) \right. .
\end{eqnarray*}

\begin{remark}
If we put $\eta =0,$ $k=0,$ and taking the limit $\rho \rightarrow 1,$ then
theorem (\ref{thm6}), reduces to theorem 5 and corollary (\ref{cor7}),
reduces to corollary 6 in \cite{ff}.
\end{remark}

\bigskip

\bigskip

Department of Mathematics Dr. Babasaheb Ambedkar Marathwada University
Aurangabad-431 004 India. E-mail: tariq10011@gmail.com.

\bigskip

Department of Mathematics Dr. Babasaheb Ambedkar Marathwada University
Aurangabad-431 004 India. E-mail: pachpatte@gmail.com.


\begin{thebibliography}{99}
\bibitem{dd} Akin E., Asl\i y\"{u}ce S., G\"{u}venilir A.F., Kaymak\c{c}alan
B. \emph{Discrete Gr\"{u}ss type inequality on fractional calculus} J.
Inequal. Appl., 2015, vol. 17 (4), (online).

\bibitem{uu} Chinchane V.L., Pachpatte D.B. \emph{a note on fractional
integral inequality involving convex functions using saigo fractional
integral} Indian journal of mathematics, 2019, vol 61 (4), no. 1,pp 27-39.

\bibitem{kk} Chinchane V.L., Pachpatte D.B. \emph{On some new Gruss-type
inequality using Hadamard fractional integral operator} J. Frac. Calc.
Appl., 2014,Vol. 5 (3S), no. 12, pp. 1-10.

\bibitem{ee} Dahmani Z., Tabharit L., and Taf S. \emph{New generalizations
of Gruss inequality using Riemann-Liouville fractional integrals}\ Bull.
Math. Anal. Appl., 2010, vol. 2 (3), pp 93--99.

\bibitem{ll} Dahmani Z. \emph{New inequalities in fractional integrals} Int.
J. Nonlinear Sci. Numer. Simul., 2010, vol. 9, no. 4, pp 493--497.

\bibitem{rr} Dragomir S. S. \emph{A generalization of Gr\"{u}ss inequality
in inner product spaces and applications} J. Math. Anal. Appl., 1999, vol.
237 (1), pp 74-82.

\bibitem{mm} Dragomir S. S. \emph{Some integral inequalities of Gruss type}
Indian J. Pur. Appl. Math., 2002, vol. 31 (4), pp 397-415.

\bibitem{tt} Elezovic N., Marangunic L. J., Pecaric J. \emph{Some
improvements of Gr\"{u}ss type inequality} J. Math. Inequal., 2007, vol. 1,
no 3,pp 425-436.

\bibitem{aa} Gruss G. \emph{Uber das Maximum des absoluten Betrages von, }$%
1/\left( b-a\right) \int_{a}^{b}f\left( t\right) g\left( t\right) dt-\left(
1/\left( b-a\right) ^{2}\right) \int_{a}^{b}f\left( t\right)
dt\int_{a}^{b}g\left( t\right) dt$ Mathematische Zeitschrift, 1935, pp
215--226.

\bibitem{hh} Katugampola U.N. \emph{A new approach to generalized fractional
derivatives} Bull. Math. Anal. Appl., 2014, vol. 6 (4), no. 1, pp 1--15.

\bibitem{jj} Katugampola U.N. \emph{New fractional integral unifying six
existing fractional integrals} pp. 6. (2016), arXiv:1612.08596 (eprint).

\bibitem{nn} McD A., Mercer, P. \emph{New proofs of the Gruss inequality}
Aust. J. Math. Anal. Appl. 1(2) (2004).

\bibitem{cc} Mitrinovic D. S., Pecaric J. E., and A. M. Fink \emph{Classical
and New Inequalities in Analysis of Mathematics and Its Applications},
Kluwer Academic, Dordrecht, The Netherlands, (1993).

\bibitem{bb} Minculete N., Ciurdariu L. \emph{A generalized form of Gr\"{u}%
ss type inequality and other integral inequalities} J. Inequal. Appl.,
(2014).

\bibitem{qq} Pachpatte B.G. \emph{A note on Chebyshev-Gruss inequalities for
differential equations} Tamsui Oxf. J. Math. Sci., 2006, vol. 22 (1), pp
29--36.

\bibitem{oo} Pachpatte B. G. \emph{On multidimensional Gruss type inetegral
inequalities} J. Inequal. Pure and Appl. Math., 2002, 1-7.

\bibitem{gg} Sousa J., Oliveira D. S., Capelas de Oliveira E. \emph{Gr\"{u}%
ss-Type Inequalities by Means of Generalized Fractional Integrals} Bull Braz
Math Soc, 2019, vol. 2 (4) (online).

\bibitem{ff} Tariboon J., Ntouyas S., and Sudsutad W. \emph{Some new
Riemann-Liouvill fractional integral inequalities} Int. J. Math. Math.
Sci.,Volume 2014, Article ID 869434, 6 pages DOI:
http://dx.doi.org/10.1155/2014/869434.

\bibitem{ss} Wang G., Agarwal P., Chand M. \emph{Certain Gr\"{u}ss type
inequalities involving the generalized fractional integral operator} J.
Inequal. Appl., ArticleI D147 (2014).

\bibitem{pp} Zhu C., Yang W., and Zhao Q. \emph{Some new fractional q-
integral Gruss-type inequalities and other inequalities} J. Inequal. Appl.
2012, 299 (online).
\end{thebibliography}
\end{document}